\documentclass{gtpart}
\usepackage{amsmath, amsthm, amssymb, amsfonts, mathtools, wasysym}
\usepackage{color}
\usepackage{verbatim}
\usepackage[mathcal]{euscript}
\usepackage{hyperref}
\usepackage[letterpaper,centering,text={5.5in,8.5in}]{geometry}
\usepackage[all]{xy}
\usepackage{graphicx}

\newtheorem{dummy}{anything}[section]
\newtheorem{theorem}[dummy]{Theorem}
\newtheorem{lemma}[dummy]{Lemma}
\newtheorem{proposition}[dummy]{Proposition}
\newtheorem{corollary}[dummy]{Corollary}
\newtheorem{definition}[dummy]{Definition}

\newtheorem{remark}[dummy]{Remark}

\newcommand{\f}[1]{\mathbb{#1}}

\newcommand{\K}{\mathbb{K}}
\newcommand{\Af}{\mathbb{A}}
\newcommand{\PP}{\mathbb{P}}

\newcommand{\Gm}{\mathbb{G}_m}
\newcommand{\g}{\mathfrak{g}}
\newcommand{\h}{\mathfrak{h}}
\newcommand{\fraksl}{\mathfrak{sl}}
\newcommand{\frakgl}{\mathfrak{gl}}

\newcommand{\calL}{\mathcal{L}}
\newcommand{\cT}{\mathcal{T}}
\newcommand{\cO}{\mathcal{O}}

\DeclareMathOperator{\Span}{span}
\DeclareMathOperator{\Div}{div}
\DeclareMathOperator{\Vect}{Vect}
\DeclareMathOperator{\End}{End}
\DeclareMathOperator{\rank}{rk}

\title{Floer cohomology of $\mathfrak{g}$-equivariant Lagrangian branes}
\author{Yank{\i} Lekili}
\address{YL: King's College London, Department of Mathematics, Strand, London WC2R 2LS, UK}
\author{James Pascaleff}
\address{JP: University of Texas at Austin, Current address: Department of Mathematics, University
of Illinois at Urbana-Champaign, 1409 W. Green St., Urbana, IL 61801, USA}
\date{}
\begin{document}

\begin{abstract}

Building on Seidel-Solomon's fundamental work \cite{seidelsol}, we define the
notion of a $\mathfrak{g}$-equivariant Lagrangian brane in an exact symplectic
manifold $M$ where $\mathfrak{g} \subset SH^1(M)$ is a sub-Lie algebra of the
symplectic cohomology of $M$. When $M$ is a (symplectic) mirror to an
(algebraic) homogeneous space $G/P$, homological mirror symmetry predicts that
there is an embedding of $\mathfrak{g}$ in $SH^1(M)$. This allows us to study a
mirror theory to classical constructions of Borel-Weil and Bott. We give explicit
computations recovering all finite dimensional irreducible representations
of $\mathfrak{sl}_2$ as representations on the Floer cohomology of an
$\mathfrak{sl}_2$-equivariant Lagrangian brane and discuss generalizations to
arbitrary finite-dimensional semisimple Lie algebras.

\end{abstract}



\maketitle

\section{Introduction}

In this paper, we are concerned with ``hidden'' symmetries on the Floer
cohomology of Lagrangian submanifolds on a symplectic manifold $X$ resulting
from an algebraic Lie group action on the mirror dual variety $X^{\vee}$. Our
work builds on and extends the work of Seidel and Solomon \cite{seidelsol} who
studied dilating $\f{C}^*$-actions on $X^{\vee}$ and interpreted these actions as symmetries on the Floer cohomology in the mirror dual $X$.

 The abstract story could be described more generally whenever $X^{\vee}$ has
an action of a semisimple Lie algebra $\g$, however for concreteness, we will
work in the setting of projective homogeneous spaces $X^{\vee}= G/P$ where $G$
is a semisimple Lie group (over $\f{C}$) and $P$ is a parabolic subgroup.
Mirror symmetry has been studied extensively in this setting. $X^{\vee}$ is
always a Fano variety. The expected A-model mirror dual to $X^{\vee}$ is a
Landau-Ginzburg model (LG-model) $W: \mathcal{R} \to \f{C}$, where $\mathcal{R}$ is an affine
variety and $W$ is a holomorphic function called the superpotential.

In the case where $G=SL_n(\f{C})$, a mirror dual LG-model of $X^{\vee}$ was first proposed in \cite{EHX} and \cite{givental} as
a superpotential $W:(\f{C}^*)^N \to \f{C}$. However, even in \cite{EHX}, it was
noticed that there was a ``disease'' with this LG-model in general. For
example, in the case $X^{\vee}= Gr(2,4)$ (Grassmannian of 2-planes in
$\f{C}^4$), one did not have the expected isomorphism \[ Jac(W) \simeq
QH^*(X^{\vee}) \] as the proposed superpotential $W: (\f{C}^*)^3 \to \f{C}$ has
only 4 critical points as opposed to $6 = \rank QH^*(X^{\vee})$. Eguchi-Hori-Xiong
suggested that to cure this disease, one has to partially compactify
$(\f{C^*})^3$. This partial compactification problem in general and the problem
of constructing an $LG$-model dual to $X^{\vee} = G/P$ for $G$ of any type was
solved by Rietsch \cite{rietsch} and the expected isomorphism of the Jacobian
ring of $W$ and $QH^*(X^{\vee})$ was obtained through an understanding of quantum cohomology established in an unpublished work
of Dale Peterson. Rietsch constructed an LG-model: \[ W : \mathcal{R} \to \f{C}
\] on an open (projected) Richardson variety $\mathcal{R}$ sitting inside the Langlands dual homogeneous variety $G^{L}/P^{L}$. This open Richardson variety is obtained as the projection from $G^{L}/B^{L}$ to $G^{L}/P^{L}$ of the intersection of two opposite Schubert cells, it is smooth and irreducible, and its complement is an anti-canonical divisor (Lemma 5.4 of \cite{KLS}).

In the case under consideration, one direction of Kontsevich's homological
mirror symmetry conjecture \cite{kontsevich} (see \cite{kontsevichENS} for
the extension to Fano varieties) states that: \begin{equation} \label{HMS} D^b
	Coh(G/P) \stackrel{?}{\simeq} D^{\pi} \mathcal{F} (\mathcal{R},W)
\end{equation} where the left hand side stands for the derived category of
coherent sheaves on the homogeneous variety $X^{\vee}=G/P$ and the right hand
side stands for the split-closed derived Fukaya category of the holomorphic
fibration $W$. Strictly speaking, a rigorous definition of the latter has only
been given in the case where $W$ has isolated non-degenerate critical points
(\cite{thebook}). This condition is equivalent to the condition that small
quantum cohomology of $G/P$ be generically semisimple. It is known that this is
the case for full flag varieties $G/B$ (\cite{kostant}), and Grassmannians
(\cite{gepner}, \cite{siebtian}). However, a counter-example in the general
case can also be found in \cite{CMP}. 

In fact, Rietsch's construction is symmetric. Namely, the Landau-Ginzburg
mirror to the homogeneous variety $X= G^{L}/P^{L}$ is an open Richardson
variety $\mathcal{R}^{\vee}$ sitting inside $X^{\vee}= G/P$ together with a
superpotential $W^{\vee} : \mathcal{R}^{\vee} \to \f{C}$. Therefore, the
expected mirror symmetry relationship can be summarized as follows (cf. \cite{aurouxanti}): 
\begin{align*} 
\mathcal{R} &\leftrightarrow \mathcal{R}^{\vee} \\
(\mathcal{R},W) &\leftrightarrow G/P \\
G^{L}/P^{L}    &\leftrightarrow (\mathcal{R}^{\vee}, W^{\vee})
\end{align*} 
where each side of the double-arrows can be considered either as an A-model or a B-model. 

On the more classical side of the story, let us recall that if $\lambda$ is a
dominant integral weight for the adjoint action of a maximal torus $T$ on $G$,
Bott-Borel-Weil theory constructs an \emph{equivariant} vector bundle
$V_\lambda$ over $G/P$ such that the space of sections $H^0(V_\lambda)$ is
isomorphic to the irreducible highest weight representation of $G$ with
highest weight $\lambda$ (\cite{bott}) . For example, in the case of
$SL_2(\f{C})$, a dominant weight is specified simply by a non-negative integer
$n \geq 0$.  Correspondingly, we have the line bundles $\mathcal{O}(n) \to
\f{C}P^1 = SL_2(\f{C}) /B$ , where $B$ is the Borel subgroup of upper-triangular
matrices in $SL_2(\f{C})$. The representations $H^0(\mathcal{O}(n))$
geometrically realize all the irreducible representations
$\text{Sym}^n(\f{C}^2)^*$ of $SL_2(\f{C})$.

If one only wishes to understand the representations of the Lie algebra $\g = \text{Lie} \  G$, then an alternative is to study the restrictions of the vector bundles
$V_\lambda$ to $\mathcal{R}^{\vee} \subset G/P$. By linearizing the $G$-action, one obtains that the space of sections, $H^0(\mathcal{R}^{\vee},V_\lambda)$, form infinite dimensional representations of $\g$, which contain, as a subspace, the finite-dimensional irreducible representation of $\g$ given by those sections that extend to $G/P$. (In
the case $G$ is simply-connected, this is all one needs in order to build all
the finite dimensional irreducible representations of $G$). Note that
$V_\lambda$ for different $\lambda$ may become isomorphic as holomorphic vector
bundles upon restriction to $\mathcal{R}^{\vee}$. However, they still would be
distinguished by their equivariant structures.  Under mirror symmetry, the process of restriction to $\mathcal{R}^\vee$ corresponds to forgetting the superpotential on $\mathcal{R}$, meaning that we consider wrapped Floer theory in $\mathcal{R}$ as developed in \cite{wrapped}.

In this paper, we will study the mirror theory to Bott-Borel-Weil theory for
$\mathcal{R}^{\vee}$ that we translate to the symplectic side $\mathcal{R}$
as inspired by the conjecture \eqref{HMS}. One of our main contributions is the
definition of a $\g$-equivariant Lagrangian where $\g \subset
SH^1(\mathcal{R})$ is a sub-Lie algebra of the symplectic cohomology of
$\mathcal{R}$ (see Definition \ref{equivariant}). 

To elaborate on this, recall that by the Hochschild-Kostant-Rosenberg theorem,
one has that: \[ HH^{*+\bullet} (G/P) \simeq H^*( \Lambda^{\bullet} (\mathcal{T}_{G/P} ))
\] where $\mathcal{T}_{G/P}$ is the tangent sheaf to $G/P$. Now, the
linearization of the action of $G$ on $G/P$ yields a map: \[ \g \to \text{Vect} (G/P)
\] which is a Lie algebra embedding since we assume that $\g$ is simple (This holds more generally whenever $G$ acts effectively on $G/P$).

Therefore, $\g$ sits inside $HH^1(G/P)$ as a sub-Lie algebra. Since Hochschild
cohomology is a derived invariant, the homological mirror symmetry conjecture
\eqref{HMS} predicts that \[ \g \subset SH^1(\mathcal{R}) \] as a sub-Lie
algebra. In Section \ref{sec:sh-vf}, we verify this prediction in the case
$G=SL_n(\f{C})$ and $\mathcal{R}= (\f{C}^*)^n$ by an explicit calculation.

The theory that we develop in Section \ref{theory} allows us to define the notion of a
$\g$-equivariant Lagrangian brane in $\mathcal{R}$ when $\g \subset
SH^1(\mathcal{R})$. The data of an equivariant structure on a Lagrangian $L$, consists of a $\f{K}$-linear map $c_L : \g \to CW^0(L,L)$ satisfying certain properties (see Definition \ref{equivariant}). For a $\g$-equivariant Lagrangian brane $L \subset \mathcal{R}$,  we use the closed-open string map to construct a
representation: \[ \rho : \g \to HW^*(L,L) \] where the latter is the wrapped
Floer cohomology of $L$. More generally, one can construct representations of
$\g$ on the wrapped Floer cohomology of a pair of $\g$-equivariant Lagrangians $(K,L)$. 

A key feature of the theory is the representations obtained this way on
$HW^*(K,L)$ depend crucially on the perturbation datum used to define various
chain level operations (cf. \cite{seidelsphere}) and the choice of equivariant
structures $c_L$ and $c_K$. In fact, the dependence on the choice of
perturbations and the equivariant structures are interrelated. In the case of
$\g= \mathfrak{sl}_2, \mathcal{R} = \f{C}^* $, we exploit this dependence in
two ways: 
\begin{enumerate}
\item In Section \ref{sec:single-lag} we fix a perturbation scheme and
consider two copies of the Lagrangian $L = \f{R}_+ \subset \f{C}^*$, one of
which is equipped with a trivial equivariant structure $c_L=0$ and the other is
equipped with a non-trivial equivariant structure $c_L$. We then construct all
the finite dimensional irreducible representations of $\mathfrak{sl}_2(\f{C})$
on a subspace of $HW^*(L,L)$. 
\item In Section \ref{sec:twists-of-L}, we take two
geometrically distinct Lagrangians that are isomorphic to $L$ in the wrapped
Fukaya category. One is the standard $L=\f{R}_+ \subset \f{C}^*$ and the other
one is $L(n)$ which is obtained from $L$ by applying $n$ times right-handed
Dehn twist to $L$ (see Figure \ref{deform}), and equip both of them with the
trivial equivariant cocycles. This amounts to picking different perturbations
in computing $HW^*(L,L)$ and by varying $n$, we again construct all the finite
dimensional irreducible representations of $\mathfrak{sl}_2(\f{C})$ on a
subspace of $HW^*(L,L)$. 
\end{enumerate}

Another interesting aspect of our theory is that the representations on
$HW^*(K,L)$ come equipped with ``canonical bases'' arising from intersections
of Lagrangian submanifolds $K$ and $L$. This is an additional piece of data
which is not apparent in Borel-Weil-Bott theory. In representation theory,
there are several bases that are called ``canonical": Lusztig's canonical bases
(\cite{lusztig1}, \cite{lusztig2}) and closely related Kashiwara's crystal
bases in quantum groups (\cite{kashiwara}), MV cycles of Mirkovic-Vilonen
\cite{mv}, etc. We will explore the relationship between our bases to these in
a future work. The relevance of canonical bases to homological mirror symmetry
has also been noticed in the work of Gross-Hacking-Keel  and
Goncharov-Shen \cite{GS}. 

\paragraph{Acknowledgments:} We would like to thank Konstanze Rietsch and
Cl\'elia Pech for helpful conversations on the paper \cite{rietsch}.  We would
also like to thank Nick Sheridan for sharing his ideas about closed-open string
maps and the $L_\infty$ structure on symplectic cohomology that proved
important in working out the results of Section 3. The notion of an equivariant
Lagrangian submanifold used in this paper relies on the work of Seidel-Solomon
\cite{seidelsol}, and we thank Paul Seidel for his help in formulating this
notion. We thank the anonymous referee for helpful comments that improved the exposition. YL was partially supported by a Royal Society Fellowship and by Marie
Curie grant EU-FP7-268389, and JP was partially supported by NSF Grant
DMS-1148490.

\section{Geometric preliminaries}

  We begin by recalling the definition of (finite-type, complete) Liouville
manifolds.  Let $(M^{cpt}, \omega = d\alpha )$ be a Liouville domain, that is,
a $2n$-dimensional compact exact symplectic manifold with boundary such that
the Liouville vector field dual to $\alpha$ points strictly outwards along
$\partial M^{cpt}$. The form $\alpha|_{\partial M^{cpt}}$ is then a contact
form. Let $R$ denote the Reeb vector field. We require that all Reeb orbits
are non-degenerate. This holds for a generic choice of $\alpha$.  Let $M$ be a
$2n$-dimensional (non-compact) symplectic manifold, obtained from the compact
domain by gluing the positive symplectization of the contact boundary: \[M = 
M^{cpt} \cup_{\partial M} [1,\infty) \times \partial M \] where, by abuse of
notation, we write $\partial M$ for $\partial M^{cpt}$.

The Liouville form $\lambda$ on the conical end is given by $\lambda = r
\alpha|_{\partial M} $ where $r$ is the coordinate in $[1, \infty)$.  We will
call $(M, d\lambda)$ constructed as above a Liouville manifold. 

On a Liouville manifold, we will consider exact properly embedded Lagrangian
submanifolds $L$ such that $\lambda$ vanishes on $L \cap (\partial M \times
[1,\infty))$. In the case $L$ is non-compact (by deforming $L$ by a Hamiltonian
isotopy if necessary) one can ensure that $L$ is of the form \[ L = L^{cpt}
\cup_{\partial L} \partial{L} \times [1,\infty) \] where $L^{cpt} \subset M^{cpt}$ and $\partial L$ is
shorthand for the Legendrian submanifold $\partial L^{cpt}$ in $\partial M$. 

In this paper, we will concern ourselves with two types of invariants of $M$.
The first one is called symplectic cohomology of $M$. This is a ``closed string
invariant'' and is a type of Hamiltonian Floer cohomology for a certain class of
Hamiltonian functions on $M$. The second one is the (wrapped) Fukaya category
of $M$ which is an ``open string invariant''. It involves (wrapped) Floer chain
complex associated with Lagrangians in $M$.

To fix an integer grading for the invariants that we will study, we impose the
following topological restrictions on $M$ and the submanifolds $L$: We will
assume that $2c_1(M)=0$ and we will fix a trivialization of the
$(\Lambda^{n}_\f{C} T^*M)^{\otimes 2}$; to define a grading for invariants
involving $L$,  we assume that the relative first Chern class $c_1(M,L) \in
H^2(M,L)$ vanishes; to define (wrapped) Fukaya category over $\f{Z}$ (rather
than $\f{Z}_2)$, we assume that all the Lagrangians that we consider are spin,
and we fix an orientation and a spin structure on $L$. 

We will henceforth assume that all these topological conditions are satisfied.
All our chain complexes will be defined over an arbitrary ring $\f{K}$, though
the one that we have in mind is principally $\f{C}$.  

\subsection{Open and closed invariants of Liouville manifolds}
\label{sec2-1}

In this section, we recall the definition of \emph{symplectic cohomology}
denoted by $SH^*(M)$, and the \emph{wrapped Fukaya category} denoted by
$\mathcal{W}(M)$. Our exposition is by no means complete. Rather, we take a
minimalistic approach to set-up the notation and refer to the literature for
more. We recommend the recent \cite{seidelsphere} for an up-to-date and
detailed account of the material summarized here.

\paragraph{Symplectic cohomology}
On Liouville manifolds we consider Hamiltonian functions $H \in
C^{\infty}(M,\f{R})$ which have linear growth outside of some compact subset of
$M$: \[ H(r, y) = m r + c \] for some $m > 0$ and $c$ are constants. We denote
the space of such functions by $\mathcal{H}(m) \subset C^{\infty}(M,\f{R})$. 

The Hamiltonian vector field (defined by the equation $-\iota_{X_H} \omega =
dH$) of such a function satisfies \[ X_H = m R \] where $R$ is the Reeb vector field.  Hence as $m$ increases
1-periodic orbits of $X_H$ correspond to longer and longer Reeb orbits where
$R$ is viewed as a vector field defined on the entire conical end via the
product structure. We require that $m$ is not equal to the period of any
Reeb orbit so that there are no $1$-periodic orbits of $H$ outside of a compact
subset of $M$. (Note that by our genericity assumptions, the Reeb vector field has
a discrete period spectrum). 

We are now ready to recall the definition of symplectic cohomology. We choose
$H \in \mathcal{H}(m)$ using a time-dependent perturbation $h : S^1 \times M
\to \f{R}$, a smooth non-negative function such that $|h|$ and
$|\lambda(X_{h})|$ are uniformly bounded.  Furthermore, we choose $h$
generically so that 1-periodic orbits of the Hamiltonian vector field $X_{H_t}$
of the function $H_t = H+ h_t$ are non-degenerate, where we write $h_t = h(t, \cdot)$. The $H_t$ perturbed action
functional on the free loop space $\mathcal{L}M$ is given by: \[ A_H(x) = -
\int x^* \lambda + \int H_t(x(t))dt \] The critical points of this functional
are $1$-periodic orbits of $X_{H_{t}}$. These give the generators of the
symplectic chain complex $SC^*(m)$ over $\f{K}$. More precisely, \[ SC^*(m) =
\bigoplus_{x \in Crit(A_H)} |o_x|_\f{K} \] where $|o_x|_\f{K}$ is the rank 1 $\f{K}$-module associated with the (real) orientation line (determinant
line) bundle (see Definition C.3 \cite{abouz} and Section 12e
of\cite{thebook}).  

The differential on $SC^*(m)$ can be defined formally via the solutions to the
negative gradient flow equation of the functional $A_H$. To be precise, one
uses the theory of $J$-holomorphic curves \'a la Floer. To ensure well-behaved
holomorphic curve theory, we will use compatible almost complex structures $J$
on $M$ which are of contact type outside of a compact subset.  This means that
\[ \lambda \circ J = dr \] It is well-known that the space of such $J$ is
contractible. The contact type condition ensures that $J$-holomorphic curves do
not escape to infinity by an application of maximum principle (\cite{biased}
pg. 10). Now, choose $S^1$-dependent family of compatible almost complex
structures $J_t$ on $M$ such that outside a compact subset of $M$, $J_t = J_0$
for all $t \in S^1$, and $J_0$ is of contact type at the conical end.

With this notation in place, the differential $d: SC^*(m) \to SC^{*+1}(m)$ is
obtained by counting finite energy solutions to the Floer's equation:
\begin{eqnarray*} u: \f{R} \times S^1 \to M \\ \partial_s u + J_t (\partial_t u
- X_{H_t}) = 0 \end{eqnarray*} Finite energy condition ensures that the limits
  $\lim_{s \to \pm \infty} u(s,\cdot)$ converges to 1-periodic orbits $x_{\pm}$
of $X_{H_t}$. 

We denote the cohomology of this chain complex by $SH^*(m)$, which is
independent of $H_t$ and $J_t$ up to canonical isomorphism. It is a finite
dimensional $\f{Z}$-graded $\f{K}$-module. Now, there are
continuation maps \[\kappa_{m}^{m'} : SC^*(m) \to SC^*(m')  \text{\  for \ } m' \geq m \] 

which are defined via an interpolation equation using $(H^s_t, J^s_t)$
depending on $s \in \f{R}$ interpolating between the perturbation datum used
for defining each group such that $\partial_s H^s_t \leq 0$. (see \cite{biased}
pg. 9 or \cite{seidelsphere} pg. 10 for a recent account). The continuation
maps are defined up to canonical isomorphism and they form a direct system,
hence one can define a chain complex via the homotopy direct limit: \[ SC^*(M) :=
\underset{m}{\text{hocolim\ }} SC^*(m) \] 
To obtain an explicit chain complex, one can adapt the model defined in \cite{wrapped}. \emph{Symplectic cohomology}, $SH^*(M)$, is the
cohomology of this chain complex.  Observe that as the colimit used in this construction is directed, one has (cf. Theorem 2.6.15 in \cite{weibel}): \[ SH^*(M) \simeq \varinjlim_m H^*(SC^*(m)) \] 
Symplectic cohomology is an algebra over the homology operad of framed little disks. We list here some operations which will
be relevant for us: \begin{itemize} \item (Product)  $\cup : SC^i(M) \otimes
SC^j(M) \to SC^{i+j}(M)$ \item (Gerstenhaber bracket)  $[ \  , \  ] : SC^i (M)
\otimes SC^j (M)  \to SC^{i+j-1} (M)$  \item (Batalin-Vilkovisky operator):  $
\Delta : SC^i(M) \to SC^{i-1} (M)$ \end{itemize}

These operations descend to $SH^*(M)$. On the cohomology level, the product is
associative and graded commutative, the latter means: \[ x \cup y = (-1)^{|x||y|} y \cup x \] The
Gerstenhaber bracket on $SH^*(M)$ satisfies $[x,y] = (-1)^{|x||y|} [y,x]$ and the graded Jacobi identity: \[
(-1)^{|x||z|} [[x,y],z]] + (-1)^{|y||x|} [[y,z],x]] +
(-1)^{|z||y|} [[z,x],y]] =0 \] In particular, note that $(SH^1(M), [ ,
])$ is a (honest) Lie algebra. 

On $SH^*(M)$, these three operations are related via the identity \footnote{We follow the sign conventions from \cite{seidelsphere}. This differs from the sign conventions in Getzler \cite{getzler}.}  : \[ [x,y] =
\Delta( x \cup y )  - \Delta(x) \cup y - (-1)^{|x|} x \cup \Delta(y)
\]
and the following Poisson (derivation) property holds:
\[ [ x, y \cup z] = [x,y] \cup z  + (-1)^{(|x|-1)|y|} y \cup [x,z] \]
We also note that there is a unital ring homomorphism: \[
H^*(M) \to SH^*(M) \] coming from the inclusion of constant orbits in
$M^{cpt}$ and $\Delta$ vanishes on the image of this inclusion.

\paragraph{Wrapped Fukaya category}

Symplectic cohomology has an open string analogue which is known as wrapped
Floer cohomology. The general construction of wrapped Fukaya category,
$\mathcal{W}(M)$, can be found in \cite{wrapped}. Here, we simply set-up the
notation.  We first recall the definition of wrapped Floer cohomology of two
exact Lagrangians $K$ and $L$. These Lagrangians need not be compact, however
if non-compact they are required to be conical at infinity as explained in the
beginning of this section.  

Choose a time dependent Hamiltonian $H_{K,L,t} : M \to \f{R}$, where the time
parameter $t$ is now in $[0,1]$. As in the closed case we require $H_{K, L,t}
\in \mathcal{H}(m)$, that is, outside of a compact subset in $M$, it is time
independent and grows linearly with some slope $m$.  The generators of (partially) wrapped
Floer complex $CF^*(K,L,m)$ is given by time-1 flow lines of the Hamiltonian
$H_{K,L}$. Concretely, these are chords $x: [0,1] \to M$ such that \[ x(0) \in K \  , \ x(1) \in L \text{
\ and \ } dx/dt = X_{H_{K,L,t}} (x) \]
We additionally require that these chords are non-degenerate (1 is not an eigenvalue of the linearization of the time 1-flow of $X_{H_{K,L,t}}$). This can be achieved by a generic (compactly supported) time dependent perturbation of $H_{K,L}$. The non-degeneracy ensures that $CF^*(K,L,m)$ is finitely generated $\f{K}$-vector space. Taking orientations into account, we write this complex as: 
\[ CF^*(K,L, m) = \bigoplus_{x} |o_x|_\f{K} \] where as before $|o_x|_\f{K}$ is the rank 1 $\f{K}$-module associated with the orientation line (determinant
line) bundle (see Section 12e of \cite{thebook}). 

The Floer differential is obtained by counts of isolated (modulo $\f{R}$
translation) finite energy maps $u: \f{R} \times [0,1] \to M$ that solve the
Floer's equation: \[ \partial_s u + J_t (\partial_t u - X_{H_{K,L,t}}) =0 \]
satisfying the boundary conditions $u( \cdot , 0 ) \in K$ and $u( \cdot,1 ) \in
L$. Here, $J_t$ is as before, a time-dependent, $\omega$-compatible complex structure on $M$
which is of contact type outside of a compact subset.

As in the closed case, one constructs canonical continuation maps: \[
CF^*(K,L,m ) \to CF^*(K,L,m') \text{ \ for \ } m' \geq m \] and defines the
wrapped Floer chain complex via the homotopy direct limit (\cite{wrapped}): \[ CW^*(K,L) :=
 \underset{m}{\text{hocolim\ }} CF^*(K,L,m) \]  The cohomology of this complex is called the
\emph{wrapped Floer cohomology} of the pair $(K,L)$.

Given a set of exact Lagrangians $L_0,\ldots, L_{d}$ in $M$, one constructs the
$A_\infty$ structure maps: \[  \mu^d : CW^*(L_{d-1}, L_d) \otimes \ldots,
\otimes CW^*(L_0,L_1) \to CW^*(L_0,L_d)[2-d] \]

by counting parametrized moduli spaces of solutions to a family of equations
analogous to Floer's equation defined on domains $\f{D} = \{ z \in \f{C} : |z|
\leq 1 \}$ with $d+1$ boundary punctures. The wrapped Fukaya category
$\mathcal{W}(M)$ has objects exact Lagrangians in $M$ (with conical ends if
non-compact) and the $A_\infty$ structure maps are given by $(\mu^d)_{d\geq 1}$
as alluded to above.  We warn the reader that the detailed construction of
these maps so as to obtain an $A_\infty$ category requires special attention
for the compatibility and consistency of perturbations used. These are well
documented in the literature to which we refer the interested reader: See
\cite{wrapped} for the first rigorous construction of wrapped Fukaya category,
\cite{abouz} for another construction and \cite{auroux} for a friendlier
discussion. Of course, all of these references build on the foundational work
in [\cite{thebook}, Section 9].

\section{Closed-Open string maps}
\label{theory}

We now recall the \emph{closed-open string map} (\cite{seideldef}) \[
\mathcal{CO}: SH^*(M) \to HH^*(\mathcal{W}(M)) \]  This is a map of (unital)
Gerstenhaber algebras. The general set-up for defining these maps is as in
Section 4 of \cite{seidelsphere}. Roughly speaking, one considers domains
$\f{D} = \{ z \in \f{C} : |z| \leq 1 \}$, with $i+1$ boundary punctures of
which $i$ are considered as inputs (and are ordered) and one interior puncture
which is considered as an input. The interior puncture additionally is equipped
with a distinguished tangent direction. One then counts isolated solutions (up
to reparametrization of the domain) of the corresponding Floer's equation such
that the interior inputs are labeled with elements of $SC^*(M)$ (the tangent
direction fixes the parametrization of the orbit) and the boundary
inputs/outputs are labeled with cochains from wrapped Floer complexes
associated with objects in $\mathcal{W}(M)$. The Floer's equation in question
is obtained by deforming the holomorphic map equation in the same way as in the
definition of $SH^*(M)$ near the interior punctures and otherwise one uses the
deformations as in the definition of $\mathcal{W}(M)$. (Note that the conformal
structure on the domain is allowed to vary.)

Properly setting up these chain level maps in a consistent manner combines two
sets of perturbations corresponding to closed and open invariants.  To spell
this out a little bit, note that one sets up the chain complex $SC^*(M)$ by
picking a class of perturbations $(H_t, J_t)$ for defining the chain complex
$SC^*(m)$ for each slope $m$ and additional data $(H^s_t, J^s_t)$ is chosen in
order to define the continuation maps $SC^*(m) \to SC^*(m+1)$ (\cite{biased},
\cite{ritter}). On the other hand, to set up the wrapped Fukaya category,
$\mathcal{W}(M)$, one chooses perturbation data for each $d-$tuple of objects
$(L_1,\dots, L_d)$ and again additional data is fixed to construct continuation
maps. The choices of perturbations are done in an inductive manner in order to
ensure consistency (\cite{wrapped}, [\cite{thebook},Section 9]).  The
consistency is required, for example, to ensure that $A_\infty$ relations hold.
The chain level operations defined by closed-open string maps combine these two
types of choices of perturbations. As a result, one has to verify the
consistency of the two sets of perturbations. This technicality is addressed in
a similar spirit to the arguments of [\cite{thebook}, Section 9]. However, in
the case where the complex structure on the domain depends on a non-compact
parameter space (which is the case for almost all the operations in this
paper), the relevant compactification of the domains (``real blow-up'' of the
Deligne-Mumford compactification from \cite{KSV}), goes slightly beyond the
case discussed in \emph{op. cit.} The way to extend this theory to this more
general case is discussed in \cite{seidelsol} and explained in more depth in
[\cite{seidelsphere}, Section 5] to which we refer the curious reader (see also
\cite{RS}).

In what follows, we fix a consistent perturbation data, which we denote by
$\mathcal{P}$ for all chain level maps and prove statements for this fixed
perturbation $\mathcal{P}$ (Note that there is a huge amount of data is
suppressed in this notation). The chain level maps always depend on the
perturbation. The dependence on perturbation often goes away when one considers
the induced maps at the cohomological level. Though, this not always the case
(crucially not so in Corollary \ref{linear} below) and we emphasize the
dependence on $\mathcal{P}$ with a subscript in such cases.

  At the chain level the $\mathcal{CO}$ map consists of an infinitude of
$\f{K}$-linear maps $\Phi^i$ indexed by the number boundary inputs. These
belong to the larger family of maps: \[ \Phi^i_j : SC^*(M)^{\otimes j} \to
hom_{\f{K}}(\mathcal{W}(M)^{\otimes i} , \mathcal{W}(M))\] where we have $j$
interior punctures as inputs and $\Phi^i_1=\Phi^i$. (One could also
think of $\Phi^i_0$ as $\mu^i$ in which case the domain is a copy of $\f{K}$ at each level of the $\f{Z}$-grading though we will not use this notation.) 

Let us now restrict the target to a subcategory consisting of two Lagrangians $K,L$. If we examine the chain level map restricted to $SC^1 (M)$, we notice the following components which will be relevant for our discussion:
\begin{eqnarray*}
\Phi_1^0 &\colon& SC^1(M) \to CW^1(K, K) \oplus CW^1(L,L) \\
\Phi_1^1 &\colon& SC^1(M) \to hom_{\f{K}} (CW^*(K,L), CW^*(K,L)) \\
\Phi_1^2 &\colon& SC^1(M) \to hom_{\f{K}} (CW^*(L,L) \otimes_{\f{K}} CW^*(K,L), CW^{*-1}(K,L)) \\ 
         &&    \hspace{1in}                   \oplus \ \ hom_{\f{K}} (CW^*(K,L) \otimes_{\f{K}} CW^*(K,K) , CW^{*-1}(K,L)) \\
\Phi_2^0 &\colon& SC^1(M) \otimes_{\f{K}} SC^1(M) \to CW^0(K,L) \\
\end{eqnarray*}
Let us emphasize again that in general the chain level maps $\Phi^i_j$ depend
on the particular perturbations used. This dependence on perturbations plays a
crucial role in this paper. 

The first one of these maps is the simplest. It is a chain map, i.e. 
\[ \Phi_1^0 d = \mu^1 \Phi_1^0 \]

For us, the most important component of the closed-open map is \[ \Phi_1^1 :
SC^1(M) \to hom_{\f{K}} (CW^*(K,L), CW^*(K,L)) \] In favorable cases, we will
use this map to define a representation of a sub-Lie algebra $\g$ of $SH^1(M)$
on $HW^*(K,L)$. To set this up, suppose that we are given a Lie
algebra embedding: $l : \g \to SH^1(M)$.
\[ \xymatrixcolsep{5pc}\xymatrix{  & SC^1(M) \ar[d] \\
            \g \lhook\mkern-7mu\ar[r]_{l} \ar@{.>}[ur]^{\tilde{l}} &  SH^1(M)
}\]

We then choose a lift $\tilde{l}: \g \to SC^1(M)$
of $l$. We prefer to do this in a pedestrian way: We choose an additive basis
$\{ [g_\alpha] \}_\alpha$ of $\g$ over $\f{K}$, where we reserve the notation
$g_\alpha$ for a choice of a cochain in $SC^1(M)$ representing $[g_\alpha] \in
SH^1(M)$. 

Next, we study the following questions in the order given via obstruction theoretical arguments:

\begin{enumerate}
\item When does $\Phi^1_1 \circ \tilde{l} $ induce a $\f{K}$-linear map: \[ \g \to End_{\f{K}} (HW^*(K,L)) \] 
\item Assuming (1) holds, when does $\Phi^1_1 \circ \tilde{l} $ induces a map of Lie algebras: \[ \g \to End_{\f{K}} (HW^*(K,L))  \] 
(Note that $End_{\f{K}}(HW^*(K,L)$ is naturally a Lie algebra using the commutator of endomorphisms). 
\item Assuming (1) and (2) hold, how does the Lie algebra representation of $\g$ on $HW^*(K,L)$ depend on the perturbation data used to define $\Phi^1_1$ and the lift $\tilde{l}$ of Lie algebra embedding $l : \g \to SH^1(M)$ ?  \end{enumerate}
It turns out that already the first question is not well-posed in general. One needs to correct the map $\Phi^1_1 \circ \tilde{l}$ by some additional terms. These modifications are due to Seidel and Solomon $\cite{seidelsol}$ which we proceed to discuss now. Once (1) (or rather a modification of it) is established, we will then answer the questions (2) and (3).

By considering the possible degenerations of the index one moduli space of
disks with one interior input, one boundary input and one boundary output where
the tangent line points towards the output boundary point, we obtain the
following:

\begin{proposition}\label{deg1} For all $x \in CW^*(K,L)$ and $a \in SC^*(M)$,
we have \[ (\Phi^1_1 d(a))(x) + (-1)^{|a|}  \Phi^1_1 (a) (\mu^1(x)) + \mu^1
(\Phi^1_1(a) (x)) =  \mu^2(\Phi^0_1(a), x)) - (-1)^{|x||a|} \mu^2(x,
\Phi^0_1(a)) \] \end{proposition} \begin{proof} See [\cite{seidelsol}, pg. 7].
The stable compactification of the moduli space of disks with one interior
point and two boundary points is homeomorphic to a closed interval. The
boundary points of this moduli space gives the term on the right hand side. The
other terms involve the differentials $d$ and $\mu^1$ and come from the
Gromov-Floer compactification of stable maps. The signs are computed as in
[\cite{seidelsol}, Section 8].  \end{proof}

Therefore, in order to induce a $\f{K}$-linear map $SH^1(M) \to End_{\f{K}}
(HW^*(K,L))$ we need to compensate for the term \[ \mu^2(x, \Phi^0_1(a)) -
(-1)^{|x|} \mu^2(\Phi^0_1(a), x)\] for all $x \in CW^*(K,L)$ and $a \in
SC^1(M)$. 

\begin{remark} Note that if $K=L$ and $\Phi^0_1(a)$ is a central element for the product, $\mu^2$, on
$HW^*(L,L)$ for all $a$, then we have:
\[ (\Phi^1_1 d(a))(x) + (-1)^{|a|}  \Phi^1_1 (a) (\mu^1(x)) + \mu^1 (\Phi^1_1(a) (x)) = 0\] 
for all $x \in CW^*(L,L)$ and $a \in SC^*(M)$.
\end{remark}

In view of the Proposition \ref{deg1}, we have the following preliminary definition (cf. \cite{seidelsol} Definition 4.2):

\begin{definition} Let $\g$ be a non-zero sub-Lie algebra of $SH^1(M)$. Choose an additive basis $\{ [g_\alpha] \}_\alpha$ of $\g$ over $\f{K}$, where we reserve the notation $g_\alpha$ for a choice of a cochain in $SC^1(M)$ representing $[g_\alpha] \in SH^1(M)$. 

A Lagrangian $L \subset M$ is \emph{$\g$-invariant} if for all $g_\alpha$, one has
\[ \Phi^0_1(g_\alpha) = \mu^1 (c_\alpha) \text{\ \ for some } \ \ c_\alpha \in CW^0(L,L)  \]

\end{definition}

We note that if $\tilde{g}_\alpha = g_\alpha + dh_\alpha$ is another lift of
$[g_\alpha]$ to $SC^1(M)$, then \[ \Phi^0_1 (\tilde{g}_\alpha) = \mu^1 (c_\alpha
+ \Phi^0_1(h_\alpha))\] as $\Phi^0_1$ is a chain map. Hence the notion of being
$\g$-invariant for a Lagrangian $L$ does not depend on the choice of lifts
$g_\alpha$, nor does it depend on the choice of basis $[g_\alpha]$. Thus, we have:

\begin{lemma}\label{obs} The obstructions to the existence of $c_\alpha$ are the classes
$[\Phi^0_1(g_\alpha)] \in HW^1(L,L)$. These classes do not depend on the perturbation data $\mathcal{P}$ and the lift $\tilde{l}$ of the embedding $l : \g \to SH^1(M)$. 
\end{lemma}

 If these classes vanish for all $g_\alpha$, then we can define a
$\f{K}$-linear map \begin{equation} c_L : \g \to CW^0(L,L) \end{equation} by
$c_L(g_\alpha) = c_\alpha$. The map $c_L$ \emph{depends} on the choice of lifts
$g_\alpha$. On the other hand, we consider $c_L$ and $c'_L$ equivalent if the image of
$c_L-c'_L : \g \to CW^0(L,L)$ consists of coboundaries for the Floer
differential. Then, the equivalence class of choices for $c_L$ does not depend
on the lifts $g_\alpha$, and is an affine space over $Hom_{\f{K}}(\g,
HW^0(L,L))$.

We then have:

\begin{corollary} \label{linear} Let $K$ and $L$ be $\g$-invariant Lagrangians in $M$. Pick a choice of basis $[g_\alpha]$ of $\g$ and lifts $g_\alpha$ to obtain maps $c_L, c_K$ as in the previous paragraph. Then
the map\[ \rho (g_\alpha)(x)= \Phi^1_1(g_\alpha)(x) - \mu^2 ( c_L (g_\alpha ), x ) + \mu^2 ( x , c_K (g_\alpha)) \]
induces a well-defined $\f{K}$-linear map, in particular is independent of $g_\alpha$,  
\[ \rho : \g \to End_{\f{K}} (HW^*(K,L))  \] 
\end{corollary}

\begin{proof} 
First, one has to check that if $\mu^1(x)=0$, then $\mu^1(\rho_\mathcal{P}(g_\alpha)(x))= 0$. Second, one has to check that changing $g_\alpha$ to $g_\alpha + dh_\alpha$ or $x \to x + \mu^1 y$ does not change the class $[\rho_\mathcal{P}(g_\alpha)(x))]  \in HW^*(K,L)$. All of these are direct consequences of the discussion preceding Lemma \ref{obs} and Proposition \ref{deg1}.
\end{proof}

We want to emphasize that to define $\rho$ we had to first fix a perturbation
data $\mathcal{P}$ and then choose cochain-level maps $c_K$ and $c_L$ as above.
For any two choices of $\mathcal{P}$ the corresponding affine spaces of choices
of $c_K$ and $c_L$ are canonically identified in a non-trivial way. In
particular, the map $\rho$ depends on the choice of perturbation data
$\mathcal{P}$. We often work with a fixed choice of perturbation and we write
$\rho_{\mathcal{P}}$ whenever we want to emphasize the dependence on the
perturbation $\mathcal{P}$, though note that this notation still suppresses the
corresponding additional choices of $c_K$ and $c_L$. 

The dependence on $\mathcal{P}$ can be controlled since the closed-open map
induces a well-defined map (independent of $\mathcal{P}$) $SH^*(M) \to
HH^*(\mathcal{W}(M))$ (\cite{seideldef}). 

Namely, let $\mathcal{P}$ and $\mathcal{P'}$ be two different perturbation data
for open and closed invariants. Restricting our attention to the wrapped Fukaya
category $\mathcal{W}(M)$ i.e. open invariants, general theory gives chain
homotopy equivalences between $CW_{\mathcal{P}}^*(K,L)$ and
$CW_{\mathcal{P'}}^*(K,L)$ for any two Lagrangians $K,L$, via continuation maps
induced by interpolation between the perturbation datum $\mathcal{P}$ and
$\mathcal{P'}$. One can analyze how the representation $\rho$ changes in a
compatible way (see for ex. Prop. \ref{prop:continuation}).  In the following
proposition, to avoid extra notational complexity, we suppress these chain
homotopy equivalences and use $CW^*(K,L)$ to denote either
$CW_{\mathcal{P}}^*(K,L)$ or $CW^*_{\mathcal{P'}} (K,L)$ by assuming that
$\mathcal{P}$ and $\mathcal{P'}$ are extensions of a fixed choice of
perturbation data used to define $\mathcal{W}(M)$. 

\begin{proposition}\label{dependence} Assume that $\mathcal{P}$ and $\mathcal{P'}$ are extensions of a fixed choice of perturbation data used to define the wrapped Fukaya category $\mathcal{W}(M)$. 
	
Then, there exist $\f{K}$-linear maps 
\[ (s_K)_{\mathcal{P}, \mathcal{P'}} : \g \to CW^0(K,K) \] 
\[ (s_L)_{\mathcal{P}, \mathcal{P'}} : \g \to CW^0(L,L) \] 
such that
\[ \rho_\mathcal{P'}(g_\alpha)(x) = \rho_{\mathcal{P}} (g_\alpha)(x) + \mu^2_\mathcal{P} ( s_L (g_\alpha),x ) + \mu^2_{\mathcal{P}} (x, s_K (g_\alpha))  \]
\end{proposition} 
\begin{proof}

Recall that the closed-open map $\mathcal{CO} : SH^*(M) \to HH^*(\mathcal{W}(M))$ is given at the chain level by the sum of the maps:
\[ \Phi^i_1 : SC^*(M) \to CC^*(\mathcal{W}(M),\mathcal{W}(M)) \]
Let us restrict our attention to the piece:
\[ \Phi^1_1: SC^1(M) \to hom(CW^* (K,L),CW^* (K,L)) \] 

Changing the perturbation data will modify this by a Hochschild coboundary of a
cochain in $CC^0(\mathcal{W}(M))$. Denote this cochain by
$(\alpha^0,\alpha^1,\ldots)$ where $\alpha^i \in hom_{\f{K}} ((CW^*)^{\otimes
i} , CW^{*-i})$. Recall that the differential on the Hochschild cochains is given by $\delta(\alpha^*) = [\alpha^*, \mu^*]$ where $[\cdot\ , \cdot ]$ is the Gerstenhaber bracket on Hochschild cochains. For the coboundary of such a chain to modify the component of
$\Phi^1_1$ that we restricted our attention to, it must be that $\alpha^0 =(s_K,s_L) \in CW^0(K,K) \oplus CW^0(L,L)$, in which case we would get:
\[ (\Phi^1_1)_\mathcal{P'} = (\Phi^1_1)_\mathcal{P} + \mu^2(s_L, \cdot) + \mu^2( \cdot , s_K) \pm \mu^1 \alpha^1 (  \cdot  ) \pm \alpha^1 \mu^1 (  \cdot  ) \]

Hence, passing to cohomology, $HW^*(K,L)$, we obtain the stated result.
\end{proof}
 
\paragraph{Summary:} Given an embedding $l : \g \to SH^1(M)$, and Lagrangians $K,L \subset M$ which are $\g$-invariant, we fix a perturbation data $\mathcal{P}$. We can then choose $c_K$ and $c_L$ and define a $\f{K}$-linear map $\rho_\mathcal{P} : \g \to End_{\f{K}} (HW^*(K,L))$. However, this map depends on $\mathcal{P}$ and the choices of $c_K$ and $c_L$. On the other hand, the overall dependence can be absorbed into the choice of the pair $c_K$ and $c_L$ which is up to equivalence is an affine space over $Hom_{\f{K}}(\g, HW^0(K,K) \oplus HW^0(L,L))$.

 \subsection{Upgrading $\rho_{\mathcal{P}}$ to a Lie algebra homomorphism}

Our next task is to upgrade $\rho_{\mathcal{P}}$ to a Lie algebra homomorphism.
This will put some conditions on the choice of $c_L$ for all $L$. To work out
what this condition is, we will need to spell out various compatibility
relations satisfied by the components of the closed-open string map and the
$A_\infty$ operations in the wrapped Fukaya category. We remark that the
relations that we discuss in this paper is a subset of the relations of the
Open-Closed Homotopy Algebra (OCHA) that was studied in \cite{KajStash}. 
  
In pseudo-holomorphic curve theory, the general strategy for obtaining relations among various counts is to consider index one moduli space of disks with fixed number of boundary and interior punctures labelled by inputs/outputs and study their degenerations. Note that interior punctures should also be labelled with a tangent direction to fix the parametrization of the orbit. 

 The next proposition and the idea of its
proof were shown to us by Nick Sheridan:

\begin{proposition}\label{lie} For all $a,b \in SC^*(M)$ and $x \in CW^*(K,L)$,
	we have: \begin{align*} & \Phi_1^1 ([a,b])(x) + (-1)^{|a|} (\Phi^1_1(a)
		\circ \Phi^1_1 (b)) (x) + (-1)^{(|a|+1)|b|} (\Phi^1_1(b) \circ
		\Phi^1_1 (a)) (x) \\ &+ (-1)^{|a|+|b|} \Phi^1_2 (a,b)( \mu^1 x)
		+ \mu^1 (\Phi^1_2(a,b)(x)) + \Phi^1_2(da,b)(x) + (-1)^{|a|}
		\Phi^1_2(a,db)(x) \\ & + \Phi^2_1 (a)(\Phi^0_1(b), x) -
		(-1)^{|b||x|} \Phi^2_1(a)(x,\Phi^0_1(b)) + (-1)^{|a||b|}
		\Phi^2_1(b)(\Phi^0_1(a),x) \\ & - (-1)^{|a|(|b|+|x|)}
		\Phi^2_1(b)(x,\Phi^0_1(a)) - \mu^2( \Phi_2^0 (a,b) ,x ) +
		(-1)^{|x| (|a|+|b|)} \mu^2(x, \Phi_2^0(a,b)) = 0 \end{align*}
\end{proposition} \begin{proof} Let us first recall how the bracket \[ [\ ,\ ] :
SC^i(M) \otimes SC^j(M) \to SC^{i+j-1}(M) \] is defined geometrically following
\cite[Sec. 4]{seidelsphere}.  Namely, one considers the moduli space of spheres
with 3 marked points equipped with tangent directions. It is easy to see that
the moduli space of all such configurations can be identified with $(S^1)^3$ -
the marked points can be uniformized to fixed locations, and the choice of
tangent direction at each marked point gives an $S^1$. Now, one chooses a
$1$-dimensional homology class in $H_1((S^1)^3)$ which rotates the tangent
directions by $2\pi$ at two marked points (inputs) in clockwise and
anticlockwise for the remaining point (output). The bracket is then defined by
counting parametrized moduli spaces of solutions to pseudo-holomorphic map
equation $M$ with labelled inputs and output where the parameter varies along an $S^1$ family of three-punctured spheres such that that tangent directions vary along the $1$-dimensional homology class that we have described.

Now, consider the degenerations of index one moduli space of disks with two
interior inputs, one boundary input, and one boundary output where tangent
directions are constrained to point towards the output boundary point. This
moduli space is smooth $1$-dimensional manifold with boundary. The various
configurations that arise at its boundary correspond to the terms in the
statement of the proposition.

The terms which involve $\mu^1$ and $d$ come from strip and cylinder breaking
in the Gromov-Floer compactification of stable maps. The terms that do not
involve the differentials $\mu^1$ or $d$ come from codimension one boundary of
the moduli space of stable disks with two boundary and two interior points as
dictated by Kimura-Stasheff-Voronov compactification \cite{KSV}. As an example,
it is worth spelling out the occurrence of the first term. This results from
the boundary contribution arising from the two interior marked points
colliding, leading to a bubbling off of a three-punctured holomorphic sphere.
As the two points collide their tangent vectors become parallel. This means
that the tangent directions in the sphere bubble are in a configuration that
lies in the 1-dimensional homology class in $(S^1)^3$ that is used to define
the bracket geometrically as explained above. Furthermore, there is an
$S^1$-worth of directions in which the two marked points come together,
parametrized by their relative phase, and this $S^1$-family covers the entire
homology class. As this three-punctured sphere has inputs labelled by $a$ and
$b$, and it outputs $[a,b]$ which is then fed into the remaining component
contributing to $\Phi^1_1$, hence the overall contribution of this boundary
component is $\Phi^1_1([a,b])(x)$.  

The other terms can be understood by similar bubbling considerations. Finally, the intricate
computation of signs follow from the discussion in [\cite{seidelsol}, Section
8].  \end{proof}

We will also need the following:

\begin{proposition}\label{multip} For all $a \in SC^*(M)$ and $x_1 \in CW^*(L_0,L_1), x_2 \in CW^*(L_1,L_2)$, we have: \begin{align*} 
&\mu^1(\Phi^2_1(a)(x_2,x_1)) - \Phi^2_1(da)(x_2,x_1) - (-1)^{|a|} \Phi^2_1(a)(\mu^1(x_2),x_1) \\ 
& - (-1)^{|a|+|x_2|} \Phi^2_1(a)(x_2,\mu^1(x_1)) + \mu^3(\Phi^0_1(a),x_2,x_1) - (-1)^{|a||x_2|} \mu^3(x_2, \Phi^0_1(a),x_1) \\ 
&+ (-1)^{(|x_2|+|x_1|)|a|} \mu^3(x_2,x_1,\Phi^0_1(a)) - \Phi^1_1(a)(\mu^2(x_2,x_1)) \\ 
&+ \mu^2(\Phi^1_1(a)(x_2),x_1) + (-1)^{(|a|+1)|x_2|} \mu^2(x_2, \Phi^1_1(a)(x_1)) = 0
\end{align*}
\end{proposition}
\begin{proof} The proof is very similar to the proof of the previous proposition. One considers the moduli space of stable disks with one interior point and three boundary points, two of which is considered as inputs and one is as output. See [\cite{seidelsol}, pg. 8]. 
\end{proof} 

Finally, by considering the moduli space of index one stable disks with two interior marked points and one boundary point, we obtain the following:

\begin{proposition} \label{second} For all $a,b \in SC^*(M)$, we have:
\begin{align*}
& \Phi^0_1([a,b]) - \Phi_1^1(a)(\Phi^0_1(b)) - (-1)^{|a||b|} \Phi_1^1(b)(\Phi^0_1(a))\\
& + \Phi^0_2(da,b) + (-1)^{|b|} \Phi^0_2(a,db) + \mu^1(\Phi^0_2(a,b)) =0 
\end{align*}
\end{proposition}

Now, as before, consider cocycles $g_\alpha, g_\beta
\in SC^1(M)$ and $x \in CW^*(K,L)$ for $\g$-invariant Lagrangians with $c_L$, $c_K$. Recall that this means $\Phi^0_1(g_\alpha) = \mu^1(c_L(g_\alpha))$ or $\mu^1(c_K(g_\alpha))$ and $\Phi^0_1(g_\beta) =\mu^1(c_L(\beta))$ or $\mu^1(c_K(g_\beta))$.
 
We are now ready to compute:
\begin{align*}
& \ \ \ \ \ \  \rho_\mathcal{P}([g_\alpha,g_\beta])(x) -
\rho_{\mathcal{P}}(g_\alpha) \circ \rho_{\mathcal{P}}(g_\beta) (x) +
\rho_{\mathcal{P}}(g_\beta) \circ \rho_{\mathcal{P}}(g_\alpha) (x) \\
& \\
&= \Phi^1_1([g_\alpha,g_\beta])(x)  - \Phi^1_1 (g_\alpha) \circ \Phi^1_1(g_\beta)(x) + \Phi^1_1(g_\beta) \circ \Phi^1_1(g_\alpha)(x) \\
&+ \Phi^1_1(g_\alpha) (\mu^2(c_L(g_\beta), x)) - \mu^2(c_L(g_\beta), \Phi^1_1(g_\alpha)(x)) -\Phi^1_1(g_\beta) (\mu^2(c_L(g_\alpha), x)) + \mu_2(c_L(g_\alpha), \Phi^1_1(g_\beta) (x)) \\
& - \Phi^1_1(g_\alpha) (\mu^2(x,c_K(g_\beta))) + \mu^2(\Phi^1_1(g_\alpha)(x), c_K(g_\beta)) + \Phi^1_1(g_\beta) (\mu^2(x,c_K(g_\alpha))) - \mu^2(\Phi^1_1(g_\beta)(x), c_K(g_\alpha)) \\ 
&- \mu^2(c_L([g_\alpha, g_\beta]),x) + \mu^2(x,c_K([g_\alpha,g_\beta])) \\ 
& -\mu^2(c_L(g_\alpha), \mu^2(c_L(g_\beta),x)) +\mu^2(c_L(g_\alpha), \mu^2(x,c_K(g_\beta)) + \mu^2(\mu^2(c_L(g_\beta),x),c_K(g_\alpha))\\ 
& - \mu^2(\mu^2(x,c_K(g_\beta)),c_K(g_\alpha)) +\mu^2(c_L(g_\beta), \mu^2(c_L(g_\alpha),x)) -\mu^2(c_L(g_\beta), \mu^2(x,c_K(g_\alpha)) \\ 
& -\mu^2(\mu^2(c_L(g_\alpha),x),c_K(g_\beta)) + \mu^2(\mu^2(x,c_K(g_\alpha)),c_K(g_\beta)) 
\end{align*}
apply Proposition \ref{lie} to get:
\begin{align*}
& = -\Phi^2_1 (g_\alpha)(\Phi^0_1(g_\beta), x) + (-1)^{|x|} \Phi^2_1(g_\alpha)(x,\Phi^0_1(g_\beta)) + \Phi^2_1(g_\beta)(\Phi^0_1(g_\alpha),x) \\ 
& - (-1)^{|x|} \Phi^2_1(g_\beta)(x,\Phi^0_1(g_\alpha)) + \mu^2( \Phi_2^0 (g_\alpha,g_\beta) ,x ) - \mu^2(x, \Phi_2^0(g_\alpha,g_\beta)) \\
&+ \Phi^1_1(g_\alpha) (\mu^2(c_L(g_\beta), x)) - \mu^2(c_L(g_\beta), \Phi^1_1(g_\alpha)(x)) -\Phi^1_1(g_\beta) (\mu^2(c_L(g_\alpha), x)) + \mu_2(c_L(g_\alpha), \Phi^1_1(g_\beta) (x)) \\
& - \Phi^1_1(g_\alpha) (\mu^2(x,c_K(g_\beta))) + \mu^2(\Phi^1_1(g_\alpha)(x),
c_K(g_\beta)) + \Phi^1_1(g_\beta) (\mu^2(x,c_K(g_\alpha))) -
\mu^2(\Phi^1_1(g_\beta)(x), c_K(g_\alpha)) \\  
&- \mu^2(c_L([g_\alpha, g_\beta]),x) + \mu^2(x,c_K([g_\alpha,g_\beta])) \\
& -\mu^2(c_L(g_\alpha), \mu^2(c_L(g_\beta),x)) +\mu^2(c_L(g_\alpha), \mu^2(x,c_K(g_\beta)) + \mu^2(\mu^2(c_L(g_\beta),x),c_K(g_\alpha))\\ 
& - \mu^2(\mu^2(x,c_K(g_\beta)),c_K(g_\alpha)) +\mu^2(c_L(g_\beta), \mu^2(c_L(g_\alpha),x)) -\mu^2(c_L(g_\beta), \mu^2(x,c_K(g_\alpha)) \\ 
& -\mu^2(\mu^2(c_L(g_\alpha),x),c_K(g_\beta)) + \mu^2(\mu^2(x,c_K(g_\alpha)),c_K(g_\beta)) + \text{coboundary.} 
\end{align*}
use $\g$-invariance of $K$ and $L$ and reorganize terms to obtain, 
\begin{align*}
& = -\Phi^2_1 (g_\alpha)(\mu^1(c_L(g_\beta)), x) + \Phi^1_1(g_\alpha) (\mu^2(c_L(g_\beta), x)) - \mu^2(c_L(g_\beta), \Phi^1_1(g_\alpha)(x)) \\
&+ \Phi^2_1(g_\beta)(\mu^1(c_L(g_\alpha)),x)  -\Phi^1_1(g_\beta) (\mu^2(c_L(g_\alpha), x)) + \mu_2(c_L(g_\alpha), \Phi^1_1(g_\beta) (x)) \\
&+ (-1)^{|x|} \Phi^2_1(g_\alpha)(x,\mu^1(c_K(g_\beta))) - \Phi^1_1(g_\alpha) (\mu^2(x,c_K(g_\beta))) + \mu^2(\Phi^1_1(g_\alpha)(x), c_K(g_\beta)) \\
&- (-1)^{|x|} \Phi^2_1(g_\beta)(x,\mu^1(c_K(g_\alpha))) + \Phi^1_1(g_\beta) (\mu^2(x,c_K(g_\alpha))) - \mu^2(\Phi^1_1(g_\beta)(x), c_K(g_\alpha)) \\
&- \mu^2(c_L([g_\alpha, g_\beta]),x) + \mu^2(x,c_K([g_\alpha,g_\beta])) \\ 
&+ \mu^2( \Phi_2^0 (g_\alpha,g_\beta) ,x ) - \mu^2(x, \Phi_2^0(g_\alpha,g_\beta)) \\
& -\mu^2(c_L(g_\alpha), \mu^2(c_L(g_\beta),x)) +\mu^2(c_L(g_\alpha), \mu^2(x,c_K(g_\beta)) + \mu^2(\mu^2(c_L(g_\beta),x),c_K(g_\alpha))\\ 
& - \mu^2(\mu^2(x,c_K(g_\beta)),c_K(g_\alpha)) +\mu^2(c_L(g_\beta), \mu^2(c_L(g_\alpha),x)) -\mu^2(c_L(g_\beta), \mu^2(x,c_K(g_\alpha)) \\ 
& -\mu^2(\mu^2(c_L(g_\alpha),x),c_K(g_\beta)) + \mu^2(\mu^2(x,c_K(g_\alpha)),c_K(g_\beta)) + \text{coboundary}. \end{align*}
use Proposition \ref{multip},
\begin{align*}
&= \mu^2(\Phi^1_1(g_\alpha)(c_L(g_\beta)),x) + \mu^3(\mu^1(c_L(g_\alpha)),c_L(g_\beta),x) \\ 
& \ \ \ \ \ \ \ \ \ \ \ \ \ \ \ \ \ \ \ \ \ \ \ \ \ \ \ \ \ \ \ \ \ \ \ \ \ \ \  - \mu^3(c_L(g_\beta),\mu^1(c_L(g_\alpha)),x) + (-1)^{|x|} \mu^3(c_L(g_\beta),x, \mu^1(c_K(g_\alpha)))   \\
& - \mu^2(\Phi^1_1(g_\beta)(c_L(g_\alpha)),x) - \mu^3(\mu^1(c_L(g_\beta)),c_L(g_\alpha),x) \\ 
&  \ \ \ \ \ \ \ \ \ \ \ \ \ \ \ \ \ \ \ \ \ \ \ \ \ \ \ \ \ \ \ \ \ \ \ \ \ \ \  + \mu^3(c_L(g_\alpha),\mu^1(c_L(g_\beta)),x) -(-1)^{|x|} \mu^3(c_L(g_\alpha),x, \mu^1(c_K(g_\beta))) \\
& -\mu^2(x,\Phi^1_1(g_\alpha)(c_K(g_\beta))) - \mu^3(\mu^1(c_L(g_\alpha)), x, c_K(g_\beta)) \\  
& \ \ \ \ \ \ \ \ \ \ \ \ \ \ \ \ \ \ \ \ \ \ \ \ \ \ \ \ \ \ \ \ \ \ \ \ \ \ \  +(-1)^{|x|} \mu^3(x,\mu^1(c_K(g_\alpha)), c_K(g_\beta)) -(-1)^{|x|} \mu^3(x,c_K(g_\beta),\mu^1(c_K(g_\alpha)))   \\
&+ \mu^2(x,\Phi^1_1(g_\beta)(c_K(g_\alpha))) + \mu^3(\mu^1(c_L(g_\beta)), x, c_K(g_\alpha)) \\
& \ \ \ \ \ \ \ \ \ \ \ \ \ \ \ \ \ \ \ \ \ \ \ \ \ \ \ \ \ \ \ \ \ \ \ \ \ \ \ -(-1)^{|x|} ( \mu^3(x,\mu^1(c_K(g_\beta)), c_K(g_\alpha)) +(-1)^{|x|} \mu^3(x,c_K(g_\alpha),\mu^1(c_K(g_\beta)))    \\
&- \mu^2(c_L([g_\alpha, g_\beta]),x) + \mu^2(x,c_K([g_\alpha,g_\beta])) \\ 
&+ \mu^2( \Phi_2^0 (g_\alpha,g_\beta) ,x ) - \mu^2(x, \Phi_2^0(g_\alpha,g_\beta)) \\
& -\mu^2(c_L(g_\alpha), \mu^2(c_L(g_\beta),x)) +\mu^2(c_L(g_\alpha), \mu^2(x,c_K(g_\beta)) + \mu^2(\mu^2(c_L(g_\beta),x),c_K(g_\alpha))\\ 
& - \mu^2(\mu^2(x,c_K(g_\beta)),c_K(g_\alpha)) +\mu^2(c_L(g_\beta), \mu^2(c_L(g_\alpha),x)) -\mu^2(c_L(g_\beta), \mu^2(x,c_K(g_\alpha)) \\ 
& -\mu^2(\mu^2(c_L(g_\alpha),x),c_K(g_\beta)) + \mu^2(\mu^2(x,c_K(g_\alpha)),c_K(g_\beta)) + \text{coboundary},
\end{align*}

\begin{align} \label{halfway} 
\notag
&= \mu^2(\Phi^1_1(g_\alpha)(c_L(g_\beta)),x) - \mu^2(\Phi^1_1(g_\beta)(c_L(g_\alpha)),x)  \\ \notag
&- \mu^2(x,\Phi^1_1(g_\alpha)(c_K(g_\beta))) + \mu^2(x,\Phi^1_1(g_\beta)(c_K(g_\alpha)))  \\ \notag
&- \mu^2(\mu^2(c_L(g_\alpha),c_L(g_\beta)),x) + \mu^2(\mu^2(c_L(g_\beta),c_L(g_\alpha)),x) \\ 
& +\mu^2(x, \mu^2(c_K(g_\alpha),c_K(g_\beta))) -\mu^2(x, \mu^2(c_K(g_\beta),c_K(g_\alpha))) \\ \notag
&- \mu^2(c_L([g_\alpha, g_\beta]),x) + \mu^2(x,c_K([g_\alpha,g_\beta])) \\ \notag
&- \mu^2( \Phi_2^0 (g_\alpha,g_\beta) ,x ) + \mu^2(x, \Phi_2^0(g_\alpha,g_\beta)) + \text{coboundary},
\end{align}

where in the last equality, we repeatedly used the $A_\infty$ relation: 
\begin{align*}
&\mu^1(\mu^3(x_3,x_2,x_1)) + \mu^3(\mu^1(x_3),x_2,x_1) + (-1)^{|x_3|} \mu^3(x_3,\mu^1(x_2),x_1)\\
& + (-1)^{|x_3|+|x_2|} \mu^3(x_3,x_2,\mu^1(x_1)) = \mu^2(x_3,\mu^2(x_2,x_1)) - \mu^2(\mu^2(x_3,x_2),x_1)) 
\end{align*}

To understand the last expression that we obtained, we pause to derive another related formula. We deduce from Proposition \ref{second} that:
\begin{align*}
\Phi^0_1([g_\alpha,g_\beta]) - \Phi^1_1(g_\alpha)(\Phi^0_1(g_\beta)) + \Phi^1_1(g_\beta)(\Phi^0_1(g_\alpha)) + \mu^1(\Phi^0_2(g_\alpha,g_\beta)) = 0
\end{align*}
Using $\g$-invariance of $L$, we obtain:
\begin{align*}
\mu^1 c_L([g_\alpha,g_\beta]) - \Phi^1_1(g_\alpha)(\mu^1(c_L(g_\beta))) +\Phi^1_1(g_\beta)(\mu^1(c_L(g_\alpha))) + \mu^1(\Phi^0_2(g_\alpha,g_\beta))=0
\end{align*}
We can now apply Proposition \ref{deg1} to deduce:
\begin{align*}
&\mu^1 c_L([g_\alpha,g_\beta]) - \mu^1 \Phi^1_1(g_\alpha)(c_L(g_\beta)) + \mu^1 \Phi^1_1(g_\beta)(c_L(g_\alpha)) \\  &+ \mu^1(\mu^2(c_L(g_\alpha), c_L(g_\beta))) - \mu^1(\mu^2(c_L(g_\beta),c_L(g_\alpha)) + \mu^1(\Phi^0_2(g_\alpha,g_\beta))=0
\end{align*} 
where we also used the Leibniz rule. 

In view of this last equality, we make the definition that is most central to this paper:

\begin{definition} \label{equivariant} Let $\g$ be a non-zero sub-Lie algebra of $SH^1(M)$. Choose an additive basis $\{ [g_\alpha] \}_\alpha$ of $\g$ over $\f{K}$, where we reserve the notation $g_\alpha$ for a choice of a cochain in $SC^1(M)$ representing $[g_\alpha] \in SH^1(M)$. Let $(L,c_L)$ be a $\g$-invariant Lagrangian, i.e., for all $g_\alpha$: \[ \Phi^0_1(g_\alpha) = \mu^1 (c_L(g_\alpha))\]
We say that $(L,c_L)$ is \emph{$\g$-equivariant} if for all $g_\alpha, g_\beta$, the cocycle :
\[ c_L([g_\alpha,g_\beta])-\Phi^1_1(g_\alpha)(c_L(g_\beta)) + \Phi^1_1(g_\beta)(c_L(g_\alpha)) + \mu^2(c_L(g_\alpha),c_L(g_\beta)) - \mu^2(c_L(g_\beta),c_L(g_\alpha))+ \Phi^0_2(g_\alpha,g_\beta) \]
is a coboundary.
\end{definition}

Now, for $K$ and $L$, $\g$-equivariant Lagrangians, we can continue our previous computation from formula (\ref{halfway}), and conclude from Leibniz rule that:
\begin{align*}
& \rho_\mathcal{P}([g_\alpha,g_\beta])(x) -
\rho_{\mathcal{P}}(g_\alpha) \circ \rho_{\mathcal{P}}(g_\beta) (x) +
\rho_{\mathcal{P}}(g_\beta) \circ \rho_{\mathcal{P}}(g_\alpha) (x) \\
&= \text{\ coboundary} 
\end{align*}

Since this is of importance, we record it as:

\begin{theorem} Let $K,L$ be $\g$-equivariant Lagrangians in the sense of Definition \ref{equivariant}, then 
\[ \rho_\mathcal{P} \colon \g \to End_{\f{K}}(HW^* (K,L)) \]  is a Lie algebra homomorphism. \end{theorem}

If one assumes the somewhat unnatural condition that $\Phi^0_1(g_\alpha)$ is
identically zero (for ex. this holds if $CW^1(L,L)$ vanishes), then it follows
that $L$ can be made invariant by picking an arbitrary cocycle $c_L(g_\alpha)$
for each $g_\alpha$. In this case, each term that appear in the definition of
$\g$-equivariance are individually cocycles.  Therefore, checking
$\g$-equivariance for a pair $(L,c_L)$ is significantly simpler - each term can
be computed at the level of cohomology. In fact, in this case, it makes sense
to impose a somewhat weaker assumption on the pair $(K,L)$ in order to conclude
that $\rho_\mathcal{P}$ is a Lie algebra homomorphism. Observe that when
$\Phi^0_1(g_\alpha)$ is zero for all $g_\alpha$, it follow from Proposition
\ref{second} that $\Phi^0_2(g_\alpha,g_\beta)$ is a cocycle. Now, we can make
the following definition:

\begin{definition}\label{def:equivariant-pair} Let $(K,c_K), (L,c_L)$ be $\g$-invariant Lagrangians such that $\Phi^0_1(g_\alpha)$ is identically zero for all $g_\alpha$ for both $K$ and $L$. Then, we say $(K,L)$ is \emph{$\g$-equivariant as a pair} if for all $g_\alpha, g_\beta \in \g$ and $x \in CW^*(K,L)$, the cocycles :
\[ c_K([g_\alpha,g_\beta])-\Phi^1_1(g_\alpha)(c_K(g_\beta)) + \Phi^1_1(g_\beta)(c_K(g_\alpha)) + \mu^2(c_K(g_\alpha),c_K(g_\beta)) - \mu^2(c_K(g_\beta),c_K(g_\alpha)) \]
and,
\[ c_L([g_\alpha,g_\beta])-\Phi^1_1(g_\alpha)(c_L(g_\beta)) + \Phi^1_1(g_\beta)(c_L(g_\alpha)) + \mu^2(c_L(g_\alpha),c_L(g_\beta)) - \mu^2(c_L(g_\beta),c_L(g_\alpha)) \]
and,
\[\mu^2(\Phi^0_2(g_\alpha,g_\beta),x) - \mu^2(x,\Phi^0_2(g_\alpha,g_\beta))\]
are coboundaries.
\end{definition}

Note, of course, that the first two conditions can be satisfied by picking $c_L=c_K=0$. It follows again from formula (\ref{halfway}) that $\rho_\mathcal{P}$ is a Lie algebra homomorphism when $(K,L)$ is a $\g$-equivariant pair. 

\begin{remark} As in the previous situation, suppose that $\Phi^0_1$ is
identically zero so that $c_L(g_\alpha)$ is a cocycle for all $g_\alpha$. In
addition, suppose that the $HW^0(L,L)$ is commutative with respect to the
$\mu^2$-product. Then formula (\ref{halfway}) simplifies drastically at the level
of cohomology for $K=L$. In other words, in this case, for any choice of $c_L$,
we get a Lie algebra homomorphism. $\rho_\mathcal{P}: \g \to End_{\f{K}}(HW^0(L,L))$.
\end{remark} 

\section{Symplectic cohomology and vector fields}
\label{sec:sh-vf}

\subsection{General observations}
\label{sec:observations}

In this section we collect some general observations that indicate how natural topological structures on symplectic cohomology may correspond to structures that are also found in the representation theory of semisimple Lie algebras, such as the Cartan subalgebra and the root lattice. The propositions in this subsection are true for an arbitrary Liouville manifold $U$, and are useful in that generality, although the connection to representation theory should only be expected when $U$ is the mirror of a homogeneous space.

Let $U$ be a Liouville manifold. Let $\mathfrak{h}$ denote the image of the canonical map $H^1(U) \to SH^1(U)$. Let $\Lambda = H_1(U;\Z)$ denote the integral first homology of $U$.
\begin{proposition}
  $\mathfrak{h}$ is an abelian Lie subalgebra of $SH^1(U)$.
\end{proposition}
\begin{proof}
  This is an immediate consequence of the fact that the BV operator $\Delta$ vanishes on the image of $H^*(U)$ in $SH^*(U)$. For if $a,b \in \mathfrak{h}$, we have
  \begin{equation}
    [a,b] = \Delta(a\cup b) - \Delta(a) \cup b + a \cup \Delta(b)
  \end{equation}
  Since $a$ and $b$ are in the image of $H^1(U)$, and $a \cup b$ is in the image of $H^2(U)$, all the terms are zero.
\end{proof}

For $\alpha \in \Lambda = H_1(U;\Z)$, let $SC^*(U)_\alpha$ denote the subspace of $SC^*(U)$ spanned by periodic orbits $\gamma$ such that $[\gamma] = \alpha$.
\begin{proposition}
  The decomposition $SC^*(U) = \bigoplus_{\alpha \in \Lambda}SC^*(U)_\alpha$ is a grading by $\Lambda = H_1(U;\Z)$, such that the differential, product, BV operator, and Lie bracket are homogeneous. 
\end{proposition}
\begin{proof}
  This is an immediate consequence of the fact that the operations are defined using maps of punctured Riemann surfaces that are asymptotic to periodic orbits. For instance, a pair of pants contributing to the coefficient of $\gamma_3$ in $\gamma_1 \cup \gamma_2$ is a homology witnessing the relation $[\gamma_3] = [\gamma_1] + [\gamma_2]$.
\end{proof}

\begin{definition}
  Let $V$ be a vector space, and let $A$ be an abelian group. A \emph{relative grading} of $V$ by $A$ is a decomposition 
  \begin{equation}
    V = \bigoplus_{\beta \in I} V_\beta
  \end{equation}
  where the set $I$ indexing the summands is given the structure of an $A$-torsor. If $V$ carries a relative grading by $A$, and we are given two pure elements $v_1, v_2$, with $v_i \in V_{\beta_{i}}$, then the difference $\beta_2 - \beta_1$ is a well-defined element of $A$ called the \emph{relative grading difference between $v_1$ and $v_2$}.
\end{definition}

A natural example of a relative grading is given as follows. Let $K$ and $L$ be connected and simply connected subspaces of a connected space $U$. Let $\mathcal{P} (K,L)$ denote the space of paths starting on $K$ and ending on $L$. If we are given two paths $v_1, v_2 \in \mathcal{P} (K,L)$, we may construct a loop in $U$, first following $v_1$, then any path on $L$ joining the end point of $v_1$ to the end point of $v_2$, then following $v_2$ in the reverse direction, then any path on $K$ joining the start point of $v_2$ to the start point of $v_1$. Since $K$ and $L$ were assumed simply connected, class of this loop in $H_1(U;\Z)$ (and indeed, its free homotopy class) is well-defined. We define an equivalence relation on $\mathcal{P} (K,L)$ by declaring $v_1 \sim v_2$ if the associated class in $H_1(U;\Z)$ vanishes. Let $I$ be the set of equivalence classes. The group $H_1(U;\Z)$ acts on $I$, since we may compose a path from $K$ to $L$ with a loop based at the start point. This action is free and transitive, so $I$ is a $H_1(U;\Z)$-torsor. If $C \subset \mathcal{P}(K,L)$ is a subset (we have in mind the set of Hamiltonian chords from $K$ to $L$), and $V = \K\langle C \rangle$ is the vector space with a basis given by $C$, then $V$ admits a relative grading by $H_1(U;\Z)$.

\begin{proposition}
  Let $K$ and $L$ be connected and simply connected Lagrangians in the Liouville manifold $U$. The wrapped Floer complex $CW^*(K,L)$ admits a relative grading by $\Lambda = H_1(U;\Z)$. The differential preserves this grading, while the map
  \begin{equation}
    \Phi^1_1 : SC^*(U) \otimes CW^*(K,L) \to CW^*(K,L)
  \end{equation}
  is homogeneous with respect to the absolute grading on $SC^*(U)$ and the relative grading on $CW^*(K,L)$. That is $\Phi^1_1$ maps $SC^*(U)_\alpha \otimes CW^*(K,L)_\beta$ to $CW^*(K,L)_{\alpha+\beta}$.
\end{proposition}
\begin{proof}
  As before, this is an immediate consequence of the fact that the operations are defined by maps of Riemann surfaces. The existence of a strip joining two Hamiltonian chords witnesses that they live in the same grading component. The map $\Phi^1_1$ counts strips with a puncture, showing that the output path is homologous to the input path plus the asymptotic loop at the puncture.
\end{proof}

Now we can spell out the expected relationship to representation theory of semisimple Lie algebras.
\begin{enumerate}
\item The subalgebra $\mathfrak{h}$ that is the image of $H^1(U) \to SH^1(U)$ should, as the notation suggests, correspond to the Cartan subalgebra. 
\item The lattice $\Lambda = H_1(U;\Z)$ should correspond to the root lattice, and the grading of $SH^1(U)$ by $\Lambda$ to the grading of $\mathfrak{g}$ by the root lattice. Note that, assuming $H^1(U) \to SH^1(U)$ is injective, $\Lambda \otimes \K \cong \mathfrak{h}^*$.
\item The relative grading of $HW^*(K,L)$ by $\Lambda$ for simply connected Lagrangians $K$ and $L$ should be compared to the fact that any \emph{irreducible} representation of a semisimple Lie algebra admits a grading by the weight lattice, which is a \emph{relative grading by the root lattice}. This is to say that the difference of any two weights appearing in the same irreducible representation is a linear combination of roots.
\end{enumerate}

\subsection{Vector fields on $\PP^1$}
\label{sec:sh-cstar}
Let us recall the structure of the symplectic cohomology of $\C^*$ as a
Batalin--Vilkovisky algebra. Let $x$ denote the complex coordinate on $\C^*$,
so that $\frac{dx}{x}$ is a holomorphic one-form with simple poles at zero and
infinity. The symplectic form $\omega$ is chosen so that $\C^*$ has two
cylindrical ends symplectomorphic to $[0,\infty) \times S^1$ with $\omega =
d(rd\phi)$, where $r$ is the radial coordinate and $\phi$ is the angular coordinate on $S^1$. The symplectic cohomology $SH^*(\C^*)$ is computed using the family of Hamiltonians which are of the form $H = mr+ c$ on each end. The grading is
determined by the form $\frac{dx}{x}$. For simplicity, we work over a field  $\f{K}$ of characteristic zero. (Though, as before, we could work over arbitrary rings, such as $\f{Z}$ or $\f{Z}_2)$.

This section is an elaboration of the following theorem. It is a special case of the general relationship between the symplectic cohomology of $T^*Q$, where $Q$ is an oriented spin manifold (below $Q = T^n$), and the homology of the free loop space $H_*(\calL Q)$. With our conventions, there is an isomorphism of BV algebras $SH^\bullet(T^*Q) \cong H_{n-\bullet}(\calL Q)$, where $n = \dim_\R Q$. This is a consequence of work of Abbondandolo--Schwarz, Salamon--Weber, and Viterbo \cite{abb-schwarz,abb-schwarz-product,salamon-weber,viterbo-2}, see also the discussion in Seidel's lecture notes \cite{seidel-lectures}.
\begin{theorem}
  The symplectic cohomology of $\C^*$ is isomorphic, as a Batalin-Vilkovisky algebra, to space of polyvector fields on $\Gm$. That is to say,  $SH^p(\C^*)$ is concentrated in degrees $p = 0, 1$, and
\begin{align}
  \label{eq:sh-basis}
  SH^0(\C^*) \cong H^0(\Gm,\cO_{\Gm}) &= \K\langle z^n \mid n \in \Z \rangle\\
  SH^1(\C^*) \cong H^0(\Gm,\cT_{\Gm}) &= \K\langle \xi_n \mid n \in \Z \rangle 
\end{align}
where $z$ denotes a coordinate on $\Gm$, $\xi_n$ denotes the vector field $z^{n+1}\partial_z$, and the Batalin--Vilkovisky structure on polyvector fields is determined by the nowhere vanishing one-form $ \Omega = z^{-1}\,dz$.
\end{theorem}
We also use the notation $\theta = \xi_0 = z\partial_z$, thus $\xi_n = z^n\theta$.

In terms of the periodic orbits of the Hamiltonian, we can take a Hamiltonian equal to $mr+c$ on each end. Because this Hamiltonian does not depend on the angular coordinate $\phi$, the periodic orbits come in $S^1$ families, given by rotating the starting point, which are indexed by an integer $n$ counting how many times a periodic orbit winds around the $S^1$ factor in $\C^* \cong \R \times S^1$ (this indexing involves a choice of orientation of $S^1$). After perturbation of this Morse--Bott situation, each periodic circle breaks into two periodic orbits, one, denoted $z^n$, corresponds to the canonical element in $H^0(S^1)$, the other, denoted $\xi_n$ corresponds to an orientation class in $H^1(S^1)$.

This picture matches precisely with the free loop space homology isomorphism $SH^p(\C^*) \cong H_{1-p}(\calL S^1)$. The free loop space $\calL S^1$ is homotopy equivalent to $S^1\times \Z$. The element $z^n$ corresponds to the element in $H_1(S^1\times\{n\})$ giving an orientation, while $\xi_n$ corresponds to the canonical element in $H_0(S^1\times \{n\})$. 

The structure of a Batalin--Vilkovisky algebra on $SH^*(\C^*)$ is defined using pseudo-holomorphic maps, but it is easier to think in terms of the free loop space. The product operation is given by composing families of loops at points where their evaluations to the base coincide, and the result is that $SH^*(\C^*)$ has unit $z^0$, and is generated by $z^1$, $z^{-1}$, and $\theta = \xi_1$, and we have $z^az^b = z^{a+b}$, $\xi_n = z^n\theta$, and $\theta^2 = 0$. Thus $SH^0(\C^*)$ is the ring of Laurent polynomials, and $SH^1(\C^*)$ is a free rank-one module over that ring.

The Batalin--Vilkovisky operator $\Delta$ maps $SH^1(\C^*)$ to $SH^0(\C^*)$. In the free loop space picture, this operator corresponds to rotating the parametrization of the loop. Taking a degree $n$ loop with a fixed parametrization, if we rotate the parametrization through a $1/n$-turn, we obtain the family of all parametrizations of that loop. Thus the full rotation results in $n$ times that family, and thus:
\begin{equation}
  \label{eq:bv-rotation}
  \Delta(\xi_n) = \Delta(z^n\theta) = nz^n
\end{equation}
Observe that for negative $n$ the orientation convention is coming into play.

Recall that a Batalin--Vilkovisky algebra $(A,\cdot,\Delta)$ has a Lie bracket of degree $-1$, given by
\begin{equation}
  \label{eq:bracket}
  [a,b]= \Delta(a\cdot b) - \Delta(a)\cdot b - (-1)^{|a|} a\cdot \Delta(b)
\end{equation}
With this bracket, $A[1]$ is a graded Lie algebra, and in particular $A^1$ is a Lie algebra. Considering this bracket on $SH^1(\C^*)$, the formula becomes
\begin{equation}
  \begin{split}
    \label{eq:bracket-sh}
    [\xi_n,\xi_m] = [z^n\theta,z^m\theta] &= \Delta(z^{n+m}\theta^2) - \Delta(z^n\theta)\cdot z^m\theta + z^n\theta \cdot \Delta(z^m\theta)\\
    &= 0 - nz^n\cdot z^m\theta + z^n\theta \cdot m z^m\\
    &= (m-n) z^{n+m}\theta = (m-n)\xi_{n+m}
  \end{split}
\end{equation}
In particular, the subspace $\K\langle\xi_{-1},\xi_0,\xi_1\rangle$ is closed under the bracket, and is isomorphic to $\fraksl_2(\K)$. Indeed, if we define
\begin{equation}
  \label{eq:sl2-triple}
  e = \xi_1, \quad h = 2\xi_0, \quad f = -\xi_{-1}
\end{equation}
then we have $[h,e] = 2e$, $[h,f] = -2f$, and $[e,f] = h$.

Now let us consider the mirror of these structures. The mirror of $\C^*$ is the algebraic torus $\Gm$. On $\Gm$, we consider the cohomology of polyvector fields
$H^*(\Gm,\Lambda^{\bullet} \cT_{\Gm}) = H^0(\Gm,\cO_{\Gm}) \oplus H^0(\Gm,\cT_{\Gm})$. It consists of global functions $H^0(\Gm,\cO_{\Gm}) \cong \K[z,z^{-1}]$, and global vector fields $H^0(\Gm,\cT_{\Gm}) \cong \K[z,z^{-1}]\cdot \theta$, where
\begin{equation}
  \label{eq:theta-vector}
  \theta = z\partial_z
\end{equation}
Indeed the notation is intended to be consistent with the symplectic cohomology picture, as we shall see. The product is just the algebraic one. We find $\theta^2 = 0$ because this would live in $\Lambda^2\cT_{\Gm}$.

It remains to consider the Batalin--Vilkovisky operator. Now we must use the Calabi--Yau structure of $\Gm$, namely the algebraic volume form $\Omega = \frac{dz}{z}$. This defines an contraction operator from polyvector fields to differential forms $\iota_\Omega : H^0(\Gm,\Lambda^p\cT_{\Gm}) \to H^0(\Gm, \Omega^{1-p}_{\Gm})$, defined for a function $f$ or a vector field $\xi$ by
\begin{equation}
  \begin{split}
  \label{eq:contraction}
  \iota_\Omega(f) &= f\Omega\\
  \iota_\Omega(\xi) &= \Omega(\xi)
\end{split}
\end{equation}
On differential forms we have the de Rham differential $d$, and the Batalin--Vilkovisky operator is $\Delta = \iota_\Omega^{-1} d \iota_\Omega$. This operation is also called $\Div_\Omega$, the divergence with respect to the volume form $\Omega$. Note that $\Omega(\theta) = z^{-1}dz(z\partial_z) = 1$. Now we compute
\begin{align}
  \label{eq:bv-vectorfield}
  \iota_\Omega(z^n \theta) &= z^n\\
  d(z^n) &= nz^{n-1}\,dz = nz^n \Omega\\
  \iota_\Omega^{-1}(nz^n \Omega) &= nz^n
\end{align}
Thus $\Delta(z^n\theta) = nz^n$ in the vector field picture as well. Now the same computation as before shows that the bracket on vector fields is
\begin{equation}
  [z^n\theta,z^m\theta] = (m-n) z^{n+m} \theta
\end{equation}
This is just the differential-geometric Lie bracket on vector fields, as we verify using $\theta = z \partial_z$:
\begin{equation}
  [z^{n+1}\partial_z, z^{m+1}\partial_z] = z^{n+1}(m+1)z^m\partial_z - z^{m+1}(n+1)z^n\partial_z = (m-n)z^{n+m+1}\partial_z
\end{equation}
The $\fraksl_2$ subalgebra we found above now looks like:
\begin{equation}
  \label{eq:sl2-algebraic}
  e = z^2\partial_z, \quad h = 2z\partial_z , \quad f = -\partial_z
\end{equation}
Observe that these are precisely the vector fields on $\Gm$ that extend to the compactification $\PP^1$: $e$ vanishes to second order at zero, $h$ vanishes to first order at zero and infinity, and $f$ vanishes to second order at infinity, while any other vector field $z^{n+1}\partial_z$ will have a pole at one of these points.

\subsection{Vector fields on $\PP^r$}
\label{sec:n-dim-case}

The prior discussion has a simple generalization to $\PP^r$, regarded as a partial flag manifold for $\fraksl_{r+1}$. We recall the elementary description of vector fields on $\PP^r$. Represent $\PP^r$ as the quotient of $\Af^{r+1}\setminus \{0\}$ by $\Gm$. The tangent sheaf of $\Af^{r+1}$ is a free sheaf generated by the vector fields $D_i = \frac{\partial}{\partial Z_i}$ where $Z_0,\dots, Z_r$ are coordinates. With respect to the $\Gm$ action, the weight of $D_i$ is reciprocal to that of $Z_i$, and so the linear vector fields $E_{ij} = Z_iD_j$ have weight zero. The Euler vector field $E = \sum_{i=0}^r E_{ii}$ is tangent to the orbits of the $\Gm$ action. To obtain vector fields on $\PP^r$, we restrict to weight zero (= linear) vector fields, and quotient by the Euler vector field $E$:
\begin{equation}
  \label{eq:vect-pn}
  H^0(\PP^r,\cT_{\PP^r}) = \langle Z_iD_j \mid i,j = 0,\dots,r \rangle / \langle Z_0D_0 + \dots + Z_rD_r \rangle
\end{equation}
The vector fields $E_{ij} = Z_iD_j$ span a Lie algebra isomorphic to $\frakgl_{r+1}$, and the Euler vector field spans the center (= scalar matrices). Thus the Lie algebra of vector fields on $\PP^r$ is isomorphic to $\fraksl_{r+1}$.

Now let $U^\vee = \PP^r \setminus \{Z_0Z_1\cdots Z_r = 0\}$ be the complement of the coordinate hyperplanes, and consider the vector fields on $U^{\vee}$. Since $U^{\vee}$ is isomorphic to an algebraic torus $\Gm^r$, after choosing coordinates $(z_1,\dots,z_r)$ and writing $\partial_i = \frac{\partial}{\partial z_i}$, we may represent vector fields as
\begin{equation}
  H^0(U^{\vee},\cT_{U^{\vee}}) = \K[z_1^{\pm 1},\dots,z_r^{\pm 1}]\otimes \langle \partial_1,\dots,\partial_r \rangle
\end{equation}
To see how such vector fields arise from restriction, let us choose affine coordinates $z_i = Z_i/Z_0$ on $\PP^r \setminus \{Z_0 = 0\}$. This is an affine space that contains $U^{\vee}$ as the locus where all $z_i$ are nonvanishing. By testing against a germ of a function on $U^{\vee}$ pulled back to $\Af^{r+1}\setminus \{0\}$, we compute this restriction map:
\begin{align}
\label{eq:vf-restriction1}
  \text{if $i\neq 0$ and $j \neq 0$,}\quad & Z_iD_j \mapsto z_i\partial_j \\
  \text{if $i = 0$ and $j \neq 0$,} \quad & Z_0 D_j \mapsto \partial_j\\
  \text{if $i \neq 0$ and $j = 0$,} \quad & Z_i D_0 \mapsto -z_i \sum_{k=1}^r z_k\partial_k\\
\label{eq:vf-restriction2}  \text{if $i = 0$ and $j = 0$,} \quad & Z_0D_0 \mapsto -\sum_{k=1}^r z_k\partial_k
\end{align}

The natural Cartan subalgebra $\h$ is spanned by the vector fields $Z_iD_i$ for $i = 0,\dots, r$ (modulo the relation that their sum is zero), or equivalently by $\theta_i = z_i\partial_i$ for $i = 1,\dots,r$ (which are linearly independent). Note that this subalgebra $\h$ consists precisely of those vector fields that are tangent to the divisor $\{Z_0Z_1\cdots Z_r = 0\}$ that we remove in constructing the mirror.

The mirror Landau--Ginzburg model is $U = (\C^*)^r$, with complex coordinates $(x_1,\dots,x_r)$, and superpotential $W = \sum_{i=1}^r x_i + (\prod_{i=1}^r x_i)^{-1}$. It is possible to homogenize this formula by considering $(\C^*)^{r+1}$, with coordinates $(X_0,X_1,\dots,X_r)$ with superpotential $W = \sum_{i=0}^r X_i$. Then $U$ is identified with the hypersurface $\prod_{i=0}^r X_i = 1$, where we set $x_i = X_i$ and eliminate $X_0$. 

Using the fact that it is isomorphic to the free loop space homology of a torus $T^r$ (with grading $k$ replaced with $r-k$), the symplectic cohomology $SH^*(U)$ is 
\begin{equation}
  SH^*(U) \cong \K[z_1^{\pm 1}, \dots,z_r^{\pm 1}]\otimes \Lambda^*[\theta_1,\dots,\theta_r] = \K[H_1(T^r;\Z)]\otimes H^*(T^r;\K)
\end{equation}
Here we have chosen an integral basis $(\theta_1,\dots,\theta_r)$ of $H^1(T^r;\Z)$, and $(z_1,\dots,z_r)$ is the dual integral basis of $H_1(T^r;\Z)$. The subspace $1 \otimes H^*(T^r;\K)$ is the image of the natural map $H^*(U)\to SH^*(U)$. In addition to the cohomological grading, there is a grading by $H_1(T^r;\Z) \cong H_1(U; \Z)$ coming from the left tensor factor, and corresponding to the grading by free homotopy classes of loops. We use the multi-index notation $z^n = z_1^{n_1}\cdots z_r^{n_r}$. Note that the vector $n$ may be identified with an element of $H_1(U;\Z)$.

\begin{proposition}
\label{prop:sh-bv-alg}
  The symplectic cohomology $SH^*((\C^*)^r)$ is isomorphic as a Batalin--Vilkovisky algebra to the polyvector fields on $\Gm^r$, where the element $\theta_i$ corresponds to $z_i\partial_i$, and the BV operator on the latter is the divergence with respect to the volume form $\Omega = \prod_{i=1}^r \frac{dz_i}{z_i}$, multiplied by $(-1)^{\text{degree}+1}$.
\end{proposition}
\begin{proof}
   Since the BV structure is determined by the product and the BV operator, it remains to match the BV operator with the one on polyvector fields. We can do this using a Morse-Bott complex for symplectic cohomology, where we obtain a torus of periodic orbits in each homotopy class $n \in H_1(U;\Z)$. The generator $z^n$ corresponds to the top cycle on this torus, the generators $z^n\theta_1,\dots,z^n\theta_r$ correspond to $(r-1)$-cycles, and a generator such as $z^n(\theta_{i_1}\wedge \cdots \wedge \theta_{i_k})$ corresponds to an $(r-k)$-cycle. The BV operator spins these cycles along the parametrization of the loops, which acts on the torus by a translation determined by the vector $n \in H_1(U;\Z) \cong H_1(T^r;\Z)$. This spinning is Poincar\'{e} dual to contraction with $n$ acting on $H^*(T^r)$. Thus we deduce, for $\eta \in \Lambda^*(\theta_1,\dots,\theta_r)$
  \begin{equation}
    \label{eq:bv-contraction}
    \Delta(z^n\eta) = z^n(\iota_n \eta)
  \end{equation}

Now to compute the BV operator on polyvector fields. Consider a polyvector field
\begin{equation}
  \Xi = z^n (z_{i_1}\partial_{i_1})\wedge \cdots \wedge (z_{i_k}\partial_{i_k})
\end{equation}
By reordering the indices, we may assume $i_j = j$ for $j = 1,\dots,k$. This changes the volume form $\Omega$ by a power of $(-1)$, but does not affect the associated divergence operator. With this reordering done, we compute the contraction
\begin{equation}
  \iota_\Omega \Xi = z^n \prod_{i=k+1}^r \frac{dz_i}{z_i}
\end{equation}
The differential is
\begin{equation}
  d\iota_\Omega \Xi = z^n\sum_{j=1}^k n_j \frac{dz_j}{z_j}\prod_{i=k+1}^r\frac{dz_i}{z_i}
\end{equation}
Applying inverse contraction yields
\begin{equation}
  \iota_\Omega^{-1}d\iota_\Omega\Xi = z^n \sum_{j=1}^k(-1)^{j-k} n_j \prod_{i\in \{1,\dots,k\}\setminus \{j\}} z_i\partial_i
\end{equation}
And this matches with the symplectic computation of $\Delta$ after multiplication by $(-1)^{k+1}$. 
\end{proof}

We are particularly interested in degree one symplectic cohomology:
\begin{equation}
  SH^1(U) \cong \K[z_1^{\pm 1}, \dots,z_r^{\pm 1}] \otimes \langle \theta_1,\dots,\theta_r \rangle 
\end{equation}

\begin{corollary}
  The Lie algebra $SH^1(U)$ is isomorphic to $\Vect(\Gm^r)$, where $\theta_i$ is identified with $z_i\partial_i$, and $z^n\theta_i$ with $z^n\cdot z_i \partial_i$.
\end{corollary}

We will sometimes use the notation $\xi_{n,i} = z^n\theta_i$ for either an element of $SH^1(U)$ or a vector field on $\Gm^r$.

By composing the restriction map on vector fields from $\PP^r$ to $\Gm^r$ in equations \eqref{eq:vf-restriction1}--\eqref{eq:vf-restriction2} with the isomorphism in Proposition \ref{prop:sh-bv-alg}, we obtain an embedding $\fraksl_{r+1}(\K) \to SH^1(U)$. The image consists those symplectic cohomology elements that correspond to the vector fields that extend to $\PP^r$. We shall now describe what various structures in $\fraksl_{r+1}(\K)$ correspond to in symplectic cohomology. First we observe that, as expected, the Cartan subalgebra $\h \cong 1 \otimes \langle \theta_1,\dots,\theta_r \rangle$ is the image of $H^1(U) \to SH^1(U)$. 

The Lie bracket preserves the grading by $H_1(T^r;\Z)$, which is a lattice in $H_1(T^r;\K) \cong \h^*$.  In fact we can identify $H_1(T^r;\Z)$ with the $A_r$ root lattice quite naturally: the latter is the intersection of the rectangular lattice $\Z^{r+1} \subset \K^{r+1}$ with the subspace where the sum of the coordinates is zero. On the other hand, the lattice $H_1(T^r;\Z) \cong H_1(U;\Z)$ embeds into $H_1((\C^*)^{r+1};\Z)$ since $U = \{X_0X_1\cdots X_r = 1\} \subset (\C^*)^{r+1}$. A system of coordinates on $H_1((\C^*)^{r+1};\Z)$ is given by integration against the one-forms $(2\pi i)^{-1}X_i^{-1}dX_i$. The sum of those one-forms vanishes when restricted to $U$, and hence evaluates trivially on $H_1(U;\Z)$. 

With this identification, we see that $\h$ sits at zero in $H_1(T^r;\Z)$, just
as $\h$ sits at zero in the root lattice, and the other elements of
$\fraksl_{r+1}(\K) \subset SH^1(U)$ fill out an $A_r$ root system.
Explicitly, the class $z_i \in H_1(T^r;\Z)$ corresponds to the cycle on $U
\cong \{(x_1,\dots,x_r)\in (\C^*)^n\}$ where the $i$-th coordinate traces a
circle and the others are held constant at $1$. Embedding this into
$(\C^*)^{r+1}$, we find that $X_0$ traces a circle in the opposite direction.
Give $H_1((\C^*)^{r+1};\Z) \cong \Z^{r+1}$ a basis $e_0,\dots,e_r$
corresponding to integration against the one-forms $(2\pi i)^{-1}X_i^{-1}dX_i$.
Then $z_i$ is identified with the vector $e_i - e_0$ in $\Z^{r+1}$. The vector
field $z_i\partial_j = (z_i/z_j)\theta_j$ ($i,j \neq 0$) corresponds to the
element of the same notation in $SH^1(U)$, and its grading is $e_i - e_j$.
The vector field $\partial_j = z_j^{-1}\theta_j$ ($j \neq 0$) corresponds to an
element with grading $e_0 - e_j$. The vector field $-z_i \sum_{k=1}^r
z_k\partial_k = -z_i \sum_{k=1}^r \theta_k$ ($i \neq 0$) corresponds to an
element with grading $e_i - e_0$. So we find that every vector of the form $e_a
- e_b$ with $a,b \in \{0,\dots,r\}$ appears, and this is precisely the $A_r$
root system.

\section{Witt algebra representations from $SH^*((\C^*)^r)$}
\label{sec:witt-sh}

Even without considering the equivariant wrapped Fukaya category, the structure of symplectic cohomology of $(\C^*)^r$ as a Batalin--Vilkovisky algebra allows us to see a number of representations of the $r$-th Witt algebra, that is, the Lie algebra $\Vect(\Gm^r)$ of vector fields on the mirror $\Gm^r$. The results in section \ref{sec:sh-vf} show that $SH^1((\C^*)^r)$ is isomorphic to $\Vect(\Gm^r)$, and that $SH^0((\C^*)^r)$ is isomorphic to the Laurent polynomial ring $\K[z_1^{\pm 1},\dots,z_r^{\pm 1}]$, which is the ring of functions $\cO(\Gm^r)$. 

Much of the structure we will use applies to any Liouville manifold $U$. Recall that $SH^*(U)$ is a BV algebra, so the bracket (which is derived from the product and the BV operator $\Delta$) turns $SH^*(U)[1]$ into a graded Lie algebra, and in particular makes $SH^0(U)$ into a Lie module over $SH^1(U)$. 

In the case of $U = (\C^*)^r$, the Lie module structure is 
\begin{equation}
  \label{eq:lie-module}
  \begin{split}
  [z^n\theta_i, z^m] = \Delta (z^n\theta_iz^m) - \Delta(z^n\theta_i) z^m + z^n\theta_i \Delta(z^m) = \Delta(z^{n+m}\theta_i) - \Delta(z^n\theta_i)z^m \\= (n_i+m_i)z^{n+m} - n_i z^nz^m = m_iz^{n+m}
\end{split}
\end{equation}
Since we have matched $\theta_i$ to $z_i \partial_i$, this is nothing but the natural action of vector fields on functions.
In the case $U = \C^*$, this simplifies to 
\begin{equation}
  \label{eq:lie-module-1}
  [\xi_i, z^j] = \Delta (\xi_iz^j) - \Delta(\xi_i) z^j + \xi_i \Delta(z^j) = \Delta(\xi_{i+j}) - \Delta(\xi_i)z^j = (i+j)z^{i+j} - i z^iz^j = jz^{i+j}
\end{equation}
where $\xi_i = z^i\theta$ corresponds to $z^{i+1}\partial_z$.

It is possible to obtain different actions of $SH^1(U)$ on $SH^0(U)$ by taking this natural action, and adding to it a Lie algebra cocycle. Eventually, these cocycles will find their way into the definition of the cochains $c_L$ for equivariant Lagrangian branes, but we begin with an abstract discussion. If $\g$ is a Lie algebra, $M$ a vector space, and $\phi_0: \g \to \End(M)$ is a Lie algebra representation of $\g$ in $M$, we have a Chevalley--Eilenberg complex $C^*(\g,\End(M))$ of $\g$ with coefficients in $\End(M)$. If $\phi_1 : \g \to \End(M)$ is another representation, then the difference defines a one-cocycle
\begin{equation}
  \delta \phi = \phi_1- \phi_0 \in Z^1(\g, \End(M))
\end{equation}
We recall that cocycle condition for a one-cochain $\psi$ is
\begin{equation}
  \psi([x,y]) = x.\psi(y) - y.\psi(x)
\end{equation}
and the action of $\g$ on $\End(M)$ is the composition of $\phi_0$ with the adjoint action.

In our case $\g = SH^1(U)$ and $M = SH^0(U)$, so $M$ is actually an algebra. It is therefore natural to consider the map $L :M \to \End(M)$ given by left multiplication. In order for this to be a map of $\g$-modules, we need for $x \in \g$ and $a,b \in M$,
\begin{align}
  L(x.a)(b) &= [x,L(a)](b)\\
  (x.a)b &= x.(ab) - a(x.b)
\end{align}
That is to say, $\g$ acts on $M$ by derivations of the product. This always holds in the case of $SH^*(U)$, since the BV algebra axioms, as discussed in Section \ref{sec2-1}, imply that the Gerstenhaber bracket and product form an odd Poisson structure. The map $L : M \to \End(M)$ induces a map on Chevalley--Eilenberg complexes, and in particular
\begin{equation}
  Z^1(\g, M) \to Z^1(\g, \End(M))
\end{equation}

We are interested in cocycles that have a geometric origin, meaning that they arise from structures on $SH^*(U)$ that can in principle be computed from the symplectic geometry. For example, they could be expressible in terms of the BV algebra structure. The most obvious is the restriction of the BV operator $\Delta$ to $SH^1(U)$ mapping to $SH^0(U)$.
\begin{proposition}
\label{prop:bv-cocycle}
Let $A^\bullet$ be a BV algebra. Regard $A^0$ as a Lie module over the Lie algebra $A^1$. Then the BV operator $\Delta$ is a cocycle on $A^1$ with values in $A^0$:
\begin{equation}
  \Delta \in Z^1(A^1,A^0)
\end{equation}
\end{proposition}
\begin{proof}
  The cocycle condition reads, for $x,y \in A^1$, 
\begin{equation}
\label{eq:bv-cocycle}
  \Delta([x,y]) = [x,\Delta(y)] - [y,\Delta(x)]
\end{equation}
The verification of this condition uses the definition of the bracket in terms of $\Delta$ and the fact that $\Delta^2 = 0$. The left-hand side becomes
\begin{equation}
\label{eq:bv-lhs}
  \Delta\left\{\Delta(xy) - \Delta(x)y + x\Delta(y)\right\} = -\Delta(\Delta(x)y) + \Delta(x\Delta(y))
\end{equation}
The first term on the right-hand side of \eqref{eq:bv-cocycle} becomes
\begin{equation}
  \label{eq:bv-rhs}
  \Delta(x\Delta(y)) - \Delta(x)\Delta(y) + x\Delta^2(y) = \Delta(x\Delta(y)) - \Delta(x)\Delta(y)
\end{equation}
The full right-hand side is \eqref{eq:bv-rhs} minus the same expression with $x$ and $y$ swapped. That agrees with \eqref{eq:bv-lhs} once we use the graded commutativity of $A^\bullet$. 
\end{proof}

In the mirror interpretation, where $U^\vee$ is the mirror variety to $U$, equipped with a volume form $\Omega$, the condition \eqref{eq:bv-cocycle} is the fact that the divergence operator $\Div_\Omega$ defines a cocycle on the Lie algebra of vector fields $\Vect(U^\vee)$ with values in functions $\cO(U^\vee)$. The mirror interpretation suggests many other cocycles as well. For example, any differential form $\omega \in \Omega^p(U^\vee)$ defines a Lie algebra cochain $\omega \in C^p(\Vect(U^\vee),\cO(U^\vee))$. In fact, this association defines a chain map of the de Rham complex of $U^\vee$ into the Chevalley--Eilenberg complex of $\Vect(U^\vee)$ with values in $\cO(U^\vee)$ (compare the coordinate-free formula for the exterior differential to the Chevalley--Eilenberg differential). Thus, closed forms also yield cocycles:
\begin{equation}
  \Omega^p_{\text{closed}}(U^\vee) \to Z^p(\Vect(U^\vee),\cO(U^\vee))
\end{equation}
Note that the divergence cocycle is not of this form, since it is a first-order differential operation, whereas evaluation of a differential form is zeroth-order.

In terms of the BV algebra structure, the evaluation $df(x)$ for a vector field $x$ and a function $f$ corresponds to $[x,f]$. Thus for any degree element $f$ of a BV algebra we obtain a cocycle $[-,f]$ (indeed, a coboundary). More interestingly, we can also obtain cocycles from ``logarithmic one-forms'' such as $f^{-1}\,df$.
\begin{proposition}
\label{prop:log-cocycle}
  Let $A^\bullet$ be a BV algebra, and $f \in A^0$ an invertible element. Then the operator $f^{-1}[-,f] : A^1\to A^0$ is a Lie algebra cocycle.
\end{proposition}
\begin{proof}
We need to show
\begin{equation}
  \label{eq:log-cocycle}
  f^{-1}[[x,y],f] = [x,f^{-1}[y,f]] - [y,f^{-1}[x,f]]
\end{equation}
First of all, the equation $f^{-1}f = 1$ and the Poisson property implies
\begin{equation}
 0 = [x,1] = [x,f^{-1}]f + f^{-1}[x,f]
\end{equation}
so $[x,f^{-1}] = -f^{-2}[x,f]$. Applying the Poisson property to the two terms on the right hand side of \eqref{eq:log-cocycle} gives us terms $[x,f^{-1}][y,f] - [y,f^{-1}][x,f]$, which cancel out. The remaining terms in the equation are the cocycle condition for $[-,f]$, multiplied by $f^{-1}$. 
\end{proof}

 The variety $\Gm^r$ carries invertible functions $z_1,\dots,z_n$, and the algebraic volume form $\Omega = \prod_{i=1}^n\frac{dz_1}{z_1}$. These give a collection of cocycles in $H^1(\Vect(\Gm^r),\cO(\Gm^r))$
\begin{equation}
  \xi \mapsto \frac{dz_1}{z_1}(\xi),\ \dots,\ \xi \mapsto \frac{dz_r}{z_r}(\xi),\ \xi \mapsto \Div_\Omega \xi
\end{equation}
These correspond, in $H^1(SH^1((\C^*)^r),SH^0((\C^*)^r))$ to the cocycles
\begin{equation}
  \xi \mapsto z_1^{-1}[\xi,z_1],\ \dots,\ \xi \mapsto z_r^{-1}[\xi,z_r],\ \xi \mapsto \Delta(\xi)
\end{equation}

\begin{remark}
In the case of $U = (\C^*)^r$, we expect that the cocycles obtained from the BV operator and invertible elements are a complete set. We quote the following general result:
\begin{theorem}[\cite{fuks}, Theorem 2.4.11]
  Let $M$ be a smooth oriented manifold, and let $\omega_1,\dots, \omega_k$ be closed one-forms on $M$ whose cohomology classes form a basis of $H^1(M;\R)$. Let $\Div$ denote the divergence with respect to a volume form on $M$. Then $H^1(\Vect(M), C^\infty(M))$ has a basis consisting of the cohomology classes of the cocycles
  \begin{equation}
    \xi \mapsto \omega_1(\xi),\ \dots,\ \xi \mapsto \omega_k(\xi),\ \xi \mapsto \Div \xi
  \end{equation}
\end{theorem}

Our desired result is a version of this for algebraic vector fields and functions on $\Gm^r$, which is an algebraic version of $(\C^*)^r$, which is a complexification of the smooth manifold $T^n = (S^1)^n$.
\end{remark}

\subsection{The case of $\C^*$ and restriction to $\fraksl_2$}
\label{sec:restriction}

We will now specialize Propositions \ref{prop:bv-cocycle} and \ref{prop:log-cocycle} to the case where the BV algebra is $A^\bullet = SH^\bullet(\C^*)$, and see how these cocycles allow us to obtain a number of interesting representations of $SH^1(\C^*) = \Vect(\Gm)$. Recall that $\xi_n = z^n\theta = z^{n+1}\partial_z$. The BV cocycle $\Delta$ gives
\begin{equation}
  \Delta(\xi_n) = nz^n .
\end{equation}
The element $z \in SH^0(\C^*)$ is invertible, and the logarithmic cocycle $z^{-1}[-,z]$ gives
\begin{equation}
  z^{-1}[\xi_n,z] = z^n.
\end{equation}
 
Since cocycles are a linear space, we have for any scalars $\alpha$ and $\beta$ a cocycle $\alpha \Delta + \beta z^{-1}[-,z]$. Adding this cocycle to the reference representation $\phi_0$ (which is the natural action of $SH^1(\C^*)$ on $SH^0(\C^*)$ by bracket) yields a representation $\rho_{\alpha,\beta}$ of $SH^1(\C^*)$ on $SH^0(\C^*)$ given by
\begin{equation}
 \rho_{\alpha,\beta}(\xi_i)z^j = (j+\alpha i + \beta)z^{i+j}.
\end{equation}
Observe from this formula that $z^{i} \mapsto z^{i+m}$ is an isomorphism between the representations $\rho_{\alpha,\beta}$ and $\rho_{\alpha,\beta+m}$. Thus integral shifts in $\beta$ do not affect the representation up to isomorphism, though the fractional part of $\beta$ is important.

On the mirror side, these representations correspond to those obtained from densities on $\Gm$. Indeed, the Witt algebra acts on $V_{\alpha,\beta}$, the space of densities of the form 
\begin{equation}
  P(z) z^\beta(dz/z)^\alpha
\end{equation}
where $P(z)$ is a Laurent polynomial \cite{kac-raina}. These representations are reducible for the Witt algebra if and only if $\beta \in \Z$ and $\alpha$ is $0$ or $1$ \cite[p.~6]{kac-raina}. Every irreducible representation of the Witt algebra (indeed of the Virasoro algebra) in which $\xi_0$ acts semisimply with finite dimensional eigenspaces is either a subquotient of some $V_{\alpha,\beta}$, or a highest-weight or lowest-weight representation \cite{mathieu}, as originally conjectured by Kac in 1982. The irreducible subquotients of $V_{\alpha,\beta}$ are distinguished within this class by the property that all of the $\xi_0$-eigenspaces have dimension less than or equal to one \cite{kaplansky}.

The map
\begin{equation}
  \K[z,z^{-1}] \to V_{\alpha,\beta}, \quad z^j \mapsto z^jz^\beta(dz/z)^\alpha
\end{equation}
intertwines the representation $\rho_{\alpha,\beta}$ on $\K[z,z^{-1}]$ and the natural action on densities by Lie derivative. Note that $\alpha$ and $\beta$ are elements of $\K$, so the Lie derivative must be interpreted in a formal sense.

With some representations of the Witt algebra ($\cong SH^1(\C^*)$) in hand, the next step is to consider the restriction to the finite dimensional subalgebra $\fraksl_2 \cong \Span \{\xi_{-1},\xi_0,\xi_1\}$. We regard $\xi_0$ as the weighting operator, $\xi_1$ as the raising operator, and $\xi_{-1}$ as the lowering operator. Note that these conventions may differ from what is found in the literature. In this section we prove:
\begin{theorem}
  The representation $V_{\alpha,\beta}$ of the Witt algebra, when restricted to the subalgebra $\fraksl_2$, contains a finite dimensional $\fraksl_2$ submodule precisely when $\alpha$ is a non-positive half-integer, and $\alpha+\beta$ is an integer. When these conditions hold, the submodule is unique, and it has dimension $(-2\alpha + 1)$, meaning that it is isomorphic to the representation of $\fraksl_2$ in degree $-2\alpha$ homogeneous polynomials in two variables.  
\end{theorem}

\begin{proof} The general problem we are faced with is, given an infinite-dimensional representation of the Witt algebra, with finite-dimensional weight spaces, to determine whether it contains a finite-dimensional representation of $\fraksl_2$. For this, a necessary condition is that $\xi_1$ has a nontrivial kernel, and the same for $\xi_{-1}$. In the case of the representations $\rho_{\alpha,\beta}$, the weight spaces are all one-dimensional, with all weights from the set $\{j + \beta \mid j \in \Z\}$ appearing. The operator $\xi_1$ raises the weight by $1$, and it has a kernel if and only if the quantity $j + \alpha + \beta$ vanishes (meaning the weight is $j+\beta =-\alpha$). This yields the necessary condition $\alpha + \beta \in \Z$. Similarly, $\xi_{-1}$ lowers weight by one, and it has a kernel if and only if $j - \alpha + \beta$ vanishes (meaning the weight is $j+\beta = \alpha$), yielding the necessary condition $-\alpha + \beta \in \Z$. That $\alpha+ \beta$ and $-\alpha + \beta$ both be integral is equivalent to the condition that $\alpha$ and $\beta$ are half-integral, either both integral or both strictly half-integral, or put another way, that $2\alpha$ and $2\beta$ are integral and of the same parity.

Another necessary condition is that the weight space at which $\xi_1$ has a kernel, namely the space of weight $-\alpha$, must be at a higher weight than the weight space at which $\xi_{-1}$ has a kernel, namely the space of weight $\alpha$. Thus we need $-\alpha \geq \alpha$, that is, $\alpha \leq 0$. These restrict the possibilities to $\alpha \in \frac{1}{2}\Z$, $\alpha \leq 0$, $\alpha + \beta \in \Z$. Since integral shifts in $\beta$ do not change the representation up to isomorphism, we may simply take $\beta = \alpha$ in every case. With this convention, we are always looking at densities of the form $P(z)\,dz^\alpha$. 

For each $\alpha \leq 0$, with $\alpha \in \frac{1}{2}\Z$, and choosing $\beta = \alpha$, we do indeed obtain a finite-dimensional representation of $\fraksl_2$. The operator $\xi_1$ has a kernel in the weight space $j = -2\alpha$, while $\xi_{-1}$ has a kernel in the weight space $j = 0$. The weights of these vectors under $\xi_0$ are $j + \alpha$, which therefore runs from $\alpha \leq 0$ to $-\alpha \geq 0$. Since there is an essentially unique representation of $\fraksl_2$ with these weights, what we have found is none other than the irreducible representation of $\fraksl_2$ of dimension $(-2\alpha + 1)$.
\end{proof}

Let us consider the mirror interpretation of this. In the standard Borel--Weil picture, the irreducible representations of $\fraksl_2$ are found in the spaces of global sections of the line bundles $\cO_{\PP^1}(n)$ on $\PP^1$ for $n \geq 0$. To connect with the above, set $n = -2\alpha$. When we restrict the line bundles $\cO_{\PP^1}(n)$ to $\Gm \subset \PP^1$, they all become isomorphic to the structure sheaf $\cO_{\Gm}$, but they retain different infinitesimal $\fraksl_2$--equivariant structures. There is a natural reason why densities arise when we want to recover these different structures. If $X$ is a variety and $D$ is an anticanonical divisor, then the canonical bundle $K_X$ is trivial on $X\setminus D$, and any two vector bundles that differ by tensoring with the canonical bundle become isomorphic on $X \setminus D$. The divergence cocycle is what allows us to recover the different equivariant structures given by tensoring with the anticanonical bundle, and shifting the action by the divergence cocycle corresponds to multiplying the sections by the volume form $\Omega= \frac{dz}{z}$.

 In the case of $\PP^1$, the canonical bundle has a square root, which is why we must consider densities of half-integral weight in order to obtain equivariant structures corresponding to all of the line bundles on $\PP^1$. For example, the sections of the line bundle $\cO_{\PP^1}(1)$ correspond to the densities of weight $-1/2$, such as $P(z)\,dz^{-1/2}$, where $P(z)$ is a polynomial of degree at most 1. Upon restriction to $\Gm$, $P(z)$ is allowed to be a Laurent polynomial.

\section{Action on Lagrangian branes}
\label{sec:brane-action}

\subsection{Equivariant structures on a single Lagrangian}
\label{sec:single-lag}

Let $L$ denote the real positive locus $(\R_+)^r$ contained in $(\C^*)^r$. This Lagrangian submanifold is also known as a cotangent fiber of $T^*T^r$. We will apply the general theory of equivariant Lagrangian branes developed in Section \ref{theory} to this object. We will find that $L$ is $SC^1((\C^*)^r)$-invariant, and can be made $SC^1((\C^*)^r)$-equivariant in several essentially different ways, meaning that it supports several equivariant structures. When restricting to the case $r = 1$, we then recover representations of $\fraksl_2$ by taking the hom space between two copies $L$ equipped with different equivariant structures.

Recall the maps
\begin{equation}
  \Phi^0_1 : SC^*((\C^*)^r) \to CW^*(L,L)
\end{equation}
\begin{equation}
\Phi^1_1 : SC^1((\C^*)^r) \to End(CW^*(L,L)) 
\end{equation}
Since these maps are somewhat sensitive to the perturbation scheme chosen, we spell out the assumptions that we need for our computations:
\begin{enumerate}
\item \label{assump:sh} The symplectic chain complex $SC^*((\C^*)^r)$ is concentrated in degrees zero through $n$, and the differential vanishes.
\item \label{assump:hw} The wrapped Floer complex $CW^*(L,L)$ is concentrated in degree zero, and hence has vanishing differential.
\item \label{assump:comm} The ring structure on $CW^0(L,L)$ is commutative.
\item \label{assump:coiso} The map $\Phi^0_1: SC^0((\C^*)^r) \to CW^0(L,L)$ is an isomorphism of commutative rings.
\end{enumerate}
Assumption \ref{assump:sh} may be achieved by a suitable time-dependent perturbation of a Morse-Bott model for symplectic cohomology. Assumption \ref{assump:hw} is achieved by working with a Hamiltonian that is suitably convex. In light of assumption \ref{assump:hw}, the next assumption \ref{assump:comm} makes sense, since the chain complex $CW^0(L,L)$ is isomorphic to its cohomology. Then assumptions \ref{assump:comm} and \ref{assump:coiso} are computations, which follow from the results of Abouzaid on the wrapped Floer cohomology of cotangent fibers \cite{abouzaid-based}. Explicitly, $HW^0(L,L)$ is isomorphic to the based loop-space homology of $T^r$, which is the Laurent polynomial ring $\K[z_1^{\pm 1},\dots,z_n^{\pm 1}]$, justifying assumption \ref{assump:comm}. Furthermore, $SH^0((\C^*)^r)$ is isomorphic to the same ring, and the map $\Phi_1^0$ is homogeneous with respect to the $H_1((\C^*)^r;\Z)$ grading that corresponds to the natural $\Z^r$ grading on the Laurent polynomial ring. Then assumption \ref{assump:coiso} follows from the fact that a homogeneous unital homomorphism from the Laurent polynomial ring to itself must be an isomorphism.

As before, we denote by $z^n$ the generator $SC^0((\C^*)^r)$ in the homotopy class of loops representing $n \in H_1((\C^*)^r;\Z)$, and by $\xi_{n,i} = z^n\theta_i$, $i = 1,\dots, r$ the generators of $SC^1((\C^*)^r)$ in the same homotopy class. To avoid confusion, we use a subscript $L$ for the generators of $CW^0(L,L)$, so that $z^n_L$ denotes the generator corresponding to a Hamiltonian chord. Then assumption \ref{assump:coiso} means more precisely that $\Phi^0_1(z^n) = z_L^n$.

\begin{proposition}
  \label{prop:cstar-invariance}
  $L$ is $SC^1((\C^*)^r)$--invariant. An arbitrary map $c_L: SC^1((\C^*)^r) \to CW^0(L,L)$ satisfies $\mu^1\circ c_L = \Phi^0_1$.
\end{proposition}
\begin{proof}
  We must show that $\Phi^0_1$ maps $SC^1((\C^*)^r)$ to coboundaries in $CW^1(L,L)$. This is clear because the target vector space is zero by assumption \ref{assump:hw}. The second assertion holds for the same reason.
\end{proof}

\begin{proposition}
  \label{prop:cstar-liemap}
  The map $\Phi^1_1: SC^1((\C^*)^r) \to End(CW^0(L,L))$ is a map of Lie algebras. The isomorphism $\Phi^0_1 : SC^0((\C^*)^r) \to CW^0(L,L)$ intertwines the action of $SC^1((\C^*)^r)$ on $SC^0((\C^*)^r)$ with $\Phi^1_1$.
\end{proposition}
\begin{proof}
  Under assumptions \ref{assump:sh} and \ref{assump:hw}, Proposition \ref{second} reads
  \begin{equation}
    \Phi^0_1([a,b]) = \Phi^1_1(a)\Phi^0_1(b) + (-1)^{|a||b|}\Phi^1_1(b)\Phi^0_1(a) 
  \end{equation}
Now we apply this to $a = z^m\theta_i \in SC^1((\C^*)^r)$ and $b = z^n \in SC^0((\C^*)^r)$. Since $\Phi^0_1(a) = 0$, this reduces to
\begin{equation}
  \Phi^0_1([z^m\theta_i,z^n]) = \Phi^1_1(z^m\theta_i)\Phi^0_1(z^n) = \Phi^1_1(z^m\theta_i)(z_L^n)
\end{equation}
Since $[z^m\theta_i,z^n] = n_iz^{m+n}$, we find that
\begin{equation}
  nz_L^{m+n} = \Phi^1_1(\xi_m)(z_L^n)
\end{equation}
This shows that $SC^1((\C^*)^r)$ acts through $\Phi^1_1$ on the ring $CW^0(L,L)$ just as it acts by bracket on $SC^0((\C^*)^r)$, with the isomorphism $\Phi^0_1$ intertwining them, and the latter action is known to define a map of Lie algebras.
\end{proof}

\begin{remark}
  The preceding proof relies on having a known form for the closed-open map $\Phi^0_1$, but there is another argument that is more in line with the abstract theory, that relies on assumptions \ref{assump:hw} and \ref{assump:comm}. Let $\mathcal{L}$ denote the full subcategory of the wrapped Fukaya category having $L$ as its only object. The closed-open map defines a Lie map (actually part of an $L_\infty$ map)
  \begin{equation}
    \Phi_1: SH^1((\C^*)^r) \to HH^1(\mathcal{L})
  \end{equation}
  The combination of assumptions \ref{assump:hw} and \ref{assump:comm} imply that $HH^1(\mathcal{L})$ is the space of derivations of the algebra $CW^0(L,L)$, so in particular, this map actually lands in endomorphisms of $CW^0(L,L)$, and it is known to be a Lie map.
\end{remark}

We can equip the Lagrangian $L$ with various cochains $c_L$ making it $SC^1((\C^*)^r)$-invariant. We pick some reference cochain $c_0$ (which might as well be zero). The map $\Phi^1_1$ is a representation of $SC^1((\C^*)^r)$ on $CW^0(L,L)$, and if we twist it as in Corollary \ref{linear}, using the $c_0$ for both copies of $L$, we obtain
\begin{equation}
  \rho_{c_0,c_0}(\xi)(f) = \Phi^1_1(\xi)(f) - \mu^2(c_0(\xi),f) + \mu^2(f,c_0(\xi))
\end{equation}
where $\xi \in SC^1((\C^*)^r)$ and $f \in CW^0(L,L)$. We have chosen to emphasize the dependence on the cochains $c_0,c_0$. Now, by the commutativity assumption \ref{assump:comm}, the $\mu^2$ terms cancel out, and thus 
\begin{equation}
  \rho_{c_0,c_0}(\xi)(f) = \Phi^1_1(\xi)(f)
\end{equation}
 and the action on $CW^0(L,L)$ is not sensitive to the choice of $c_0$.

To change the action in a nontrivial way, we can modify the choice of $c_L$ for one copy of $L$ but not the other. Take a map $\gamma: SC^1((\C^*)^r) \to CW^0(L,L)$ , and use $c_L = c_0+\gamma$ for the first copy of $L$, and retain $c_0$ for the second copy. Then we find
\begin{equation}
  \rho_{c_0+\gamma,c_0}(\xi)(f) = \rho_{c_0,c_0}(\xi)(f) + \mu^2(f,\gamma(\xi)) = \Phi^1_1(\xi)(f) + \mu^2(f,\gamma(\xi))
\end{equation}
A similar shift happens if we modify the $c_L$ for the second copy of $L$. This is summarized as the following proposition.
\begin{proposition}
  \label{prop:cstar-actions}
  Let $c_1$ and $c_2$ be two maps $SC^1((\C^*)^r) \to CW^0(L,L)$. Using $c_1$ for the first copy of $L$, and $c_2$ for the second copy of $L$, the linear map $\rho_{c_1,c_2}$ of $SC^1((\C^*)^r)$ on $CW^0(L,L)$ is given by
  \begin{equation}
    \rho_{c_1,c_2}(\xi)(f) = \Phi^1_1(\xi)(f) + \mu^2(c_1(\xi)-c_2(\xi), f)
  \end{equation}
  The map $\rho_{c_1,c_2}: SC^1((\C^*)^r) \to End(CW^0(L,L))$ is a Lie map if and only if $c_1-c_2$ is a cocycle in the Chevalley--Eilenberg complex of $SC^1((\C^*)^r)$ with values in the module $CW^0(L,L)$, where the latter space is regarded as a module via $\Phi^1_1$.
\end{proposition}
\begin{proof}
  The expression for $\rho_{c_1,c_2}$ follows from the discussion preceding the proposition and the commutativity assumption \ref{assump:comm}. The second assertion is standard (and we already used it in the discussion of cocycles on $SH^1((\C^*)^r)$).
\end{proof}

\begin{remark}
  The fact that the difference $c_1-c_2$ should be a Lie algebra cocycle is already evident in the definition of an equivariant Lagrangian. Under assumptions \ref{assump:sh} and \ref{assump:hw} and \ref{assump:comm}, the equation that $c_L$ ought to solve is
  \begin{equation}
    c_L([x,y]) - \Phi^1_1(x)(c_L(y)) + \Phi^1_1(y)(c_L(x)) + \Phi^0_2(x,y) = 0
  \end{equation}
  and the first three terms on the left hand side are the Chevalley--Eilenberg differential of $c_L$. Therefore, the difference of two solutions of this equation is a Chevalley--Eilenberg cocycle.
\end{remark}

We can obtain cocycles $\gamma_L: SC^1((\C^*)^r) \to CW^0(L,L)$ using the results of Section \ref{sec:witt-sh}. Let $\gamma \in Z^1(SC^1((\C^*)^r),SC^0((\C^*)^r))$ be a Lie algebra cocycle, for example we could take the cocycles considered before
\begin{equation}
  \gamma = \alpha \Delta + \beta_1 z_1^{-1}[-,z_1] + \cdots + \beta_r z_r^{-1}[-,z_r]
\end{equation}
and set
\begin{equation}
  \gamma_L = \Phi^0_1\circ \gamma : SC^1((\C^*)^r) \to CW^0(L,L)
\end{equation}
Now because $\Phi^0_1$ is an isomorphism (assumption \ref{assump:coiso}) that intertwines the actions of $SC^1((\C^*)^r)$ on $SC^0((\C^*)^r)$ and $CW^0(L,L)$ (Proposition \ref{prop:cstar-liemap}), it induces an isomorphism between the spaces of Lie algebra cocycles with values in these two modules. Thus $\gamma_L$ is indeed a cocycle. 

For example, in the case of $\C^*$, we have the following
\begin{proposition}
  Choose $c_1$ and $c_2$ such that 
  \begin{equation}
    c_1-c_2 = \gamma_L = \Phi^0_1\circ( \alpha \Delta + \beta z^{-1}[-,z])
  \end{equation}
  Then the map $\rho_{c_1,c_2}$ is a Lie algebra representation of $SC^1(\C^*)$ on $CW^0(L,L)$. This representation is isomorphic to the representation $V_{\alpha,\beta}$ of scalar densities of the form $P(z) z^\beta (dz/z)^\alpha$
\end{proposition}
\begin{proof}
  The discussion preceding the proposition makes the first assertion clear. The second is a consequence of the discussion in Section \ref{sec:witt-sh} and the fact that $\Phi^0_1$ is an isomorphism by assumption 4. 
  \end{proof}

We can restrict our whole discussion to the subalgebra $\fraksl_2$ inside
$SC^1(\C^*)$, and we find once again, that when $\alpha$ is a non-positive
half-integer, and $\alpha+\beta$ is integral, the representation of $\fraksl_2$
on $CW^0(L,L)$ contains a unique finite-dimensional submodule, which is
isomorphic to the $(-2\alpha+1)$-dimensional irreducible representation of
$\fraksl_2$.

\label{sec:twists-of-L}

We can obtain all representations of $\fraksl_2$ by placing different equivariant structures on the single Lagrangian $L$, but one could object that this process is somewhat artificial: the representations actually arise inside a representation of the Witt algebra in densities, obtained by choosing a rather artificial cochain $c_L$, and taking the restriction to $\fraksl_2$. In a sense, once we find the representation of the Witt algebra in Laurent polynomials, the rest is an algebraic formality. In this section we will see a way in which the finite-dimensional representations of $\fraksl_2$ are more naturally implied by the geometry of $\C^*$ and its Lagrangian submanifolds. 

\begin{figure}[!h] \centering
\includegraphics[scale=0.6]{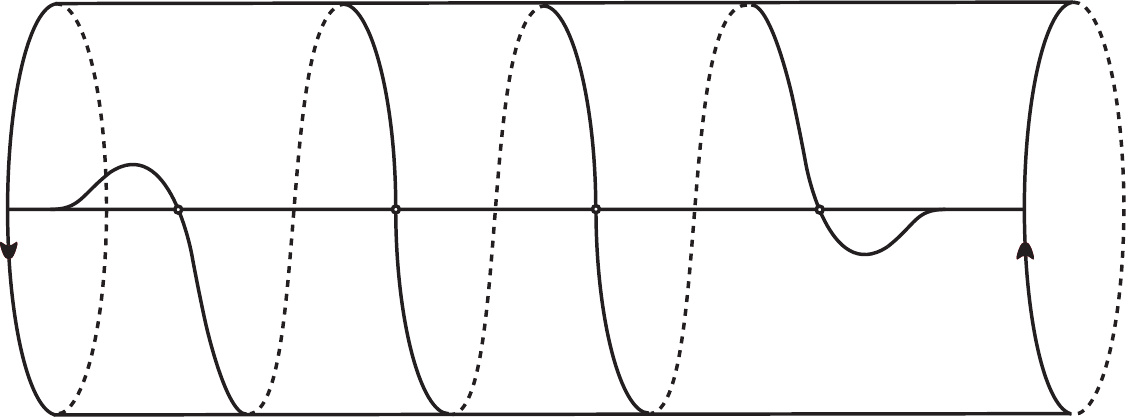}  \caption{The Lagrangian $L(3)$; the vector space generated by the 4 intersection points gives the 4-dimensional irreducible representation of $\mathfrak{sl}_2$ . }
\label{deform} \end{figure}

We begin by summarizing the construction. Start with the Lagrangian $L$ (the real positive locus) considered above. Then, for each $n \in \Z$, we construct another Lagrangian $L(n)$ which is the image of $L$ under the $n$-th power of the Dehn twist about the core circle $S^1 \subset \C^*$ (the zero section in $T^*S^1$). Thus $L = L(0)$. Now, in contrast to the previous construction, we equip all Lagrangians with cochains $c_{L(n)} \equiv 0$. It turns out that, with this choice, and for $n \geq 0$, $HW^0(L,L(n))$ becomes a representation of $SH^1(\C^*)$. Next, we look more closely at the generators of $HW^0(L,L(n))$, which are either intersection points of $L$ and $L(n)$, or chords from $L$ to $L(n)$. Let $V(n) \subset HW^0(L,L(n))$ be the subspace spanned by the intersection points (which is to say, not including the chords of positive length). We show, using geometric arguments, that $V(n)$ is stable under $\fraksl_2$, and furthermore that its weight spaces (under the action of the Cartan subalgebra $H^1(\C^*) \subset \fraksl_2 \subset SH^1(\C^*)$ ) are one-dimensional. This determines $V(n)$ as the irreducible representation of highest weight $n$. The details follow.

To construct $L(n)$, we use the Dehn twist about the core circle $S^1 = \{|z| = 1\} \subset \C^*$. This is a compactly supported symplectic automorphism of $\C^*$, which is unique up to isotopy, and we choose a representative $\tau$ that is supported in a particular annulus $A \subset \C^*$. We regard the choice of $A$ as decomposing $\C^*$ into several parts: an \emph{interior} given by $A$ itself, and two \emph{ends}, the components of $\C^* \setminus A$.  We define $L(n) = \tau^n L$, and we call it the $n$-th twist of $L$. Note well that, on the ends of $\C^*$, $L$ and $L(n)$ coincide, as we will make use of this later. With respect to the decomposition, we choose a Hamiltonian to define the wrapped complexes $CW^*(L,L(n))$ that is small (essentially zero) on the interior, and grows on the ends. The wrapped complex $CW^*(L,L(n))$ has generators given by chords of the Hamiltonian flow starting on $L$ and ending on $L(n)$. Each chord has an action, which is roughly a measure of its length. Since we assume that the wrapping Hamiltonian is small in the interior, each intersection point between $L$ and $L(n)$ gives rise to a ``short'' chord with small action. By an abuse of language we call these generators in $CW^*(L,L(n))$ \emph{intersection generators}. Also, on the end where $L$ and $L(n)$ coincide, we obtain chords of various actions, the smallest of which corresponds to an intersection between $L$ and a small perturbation of $L(n)$ on that end. We declare these generators to be \emph{intersection generators} as well. The other chords, which ``go all the way around'' $\C^*$, we call \emph{proper chord generators}.  With an appropriate choice of $\tau$ and the wrapping Hamiltonian $H$, we can ensure the following assumption in addition to assumptions \ref{assump:sh}--\ref{assump:coiso}:
\begin{enumerate}
\setcounter{enumi}{4}
\item \label{assump:hw-l-ln}If $n > 0$, then all generators of $CW^*(L,L(n))$ lie in degree zero, hence the differential vanishes.
\end{enumerate}

  In the case where $n < 0$, this assumption will not be possible to satisfy, and in the case $n = 0$, it would conflict with assumption \ref{assump:hw}. When computing $CW^*(L,L)$ we want to use a Hamiltonian that creates only one intersection generator, namely the element $1_L \in CW^0(L,L)$, where as the above prescription would create two degree zero intersection generators, one in each end, and hence an odd generator in the interior to compensate, giving the same cohomology.

We also need to use the wrapped complex $CW^*(L(n),L(n))$, and for this we use the same perturbation scheme as for $L$, in the sense that we take the image of this scheme under the automorphism $\tau^n$. With this choice, we obtain
\begin{enumerate}
\setcounter{enumi}{5}
\item \label{assump:hw-ln} 
$CW^*(L(n),L(n))$ is concentrated in degree zero, and 
  \begin{equation}
    (\tau^n)_* : CW^0(L,L) \to CW^0(L(n),L(n))
  \end{equation}
  is a ring isomorphism.
\end{enumerate}

With this setup in place, we begin our derivation. Fix some $n > 0$.
\begin{proposition}
  Both $L$ and $L(n)$ are $SC^1(\C^*)$--invariant, with cochains $c_L = 0$ and $c_{L(n)} = 0$.
\end{proposition}
\begin{proof}
  Clear by assumptions \ref{assump:hw} and \ref{assump:hw-l-ln}.
\end{proof}

The choice $c_L=0$ makes $L$ equivariant in the sense of Definition \ref{equivariant} if only if the terms $\Phi^0_2(g_\alpha,g_\beta)$ vanishes. This obstruction is potentially different for every object, so we include the object as a subscript. There is a relationship between $\Phi^0_{2,L}$ and $\Phi^0_{2,L(n)}$.
\begin{lemma}
  \label{lem:same-phi}
  The isomorphism $(\tau^n)_* : CW^0(L,L) \to CW^0(L(n),L(n))$ intertwines the maps $\Phi^0_{2,L}$ and $\Phi^0_{2,L(n)}$:
  \begin{equation}
    (\tau^n)_* \circ \Phi^0_{2,L} = \Phi^0_{2,L(n)} : SC^1(\C^*)^{\otimes 2} \to CW^0(L(n),L(n))
  \end{equation}
\end{lemma}
\begin{proof}
  Here we use the fact that the symplectic automorphism $\tau$ acts on $SC^1(\C^*)$ by the identity map, for the simple reason that it is compactly supported, so doesn't affect the generators in the ends, and it acts by identity on the cohomology of the interior. Thus, if we apply the map $\tau^n$ to the moduli space of curves computing the coefficient of $x \in CW^0(L,L)$ in $\Phi^0_{2,L}(a,b)$, we obtain a moduli space that computes the coefficient of $(\tau^n)_*(x) \in CW^0(L(n),L(n))$ in $\Phi^0_{2,L(n)}(a,b)$.
\end{proof}

\begin{proposition}
  The pair $(L,L(n))$ is equivariant as a pair (Definition \ref{def:equivariant-pair}). Thus the map $(\Phi^1_1)_{L,L(n)} : SC^1(\C^*) \to \End(CW^0(L,L(n)))$ is a map of Lie algebras.
\end{proposition}
\begin{proof}
  Taking $c_L=0$ and $c_{L(n)}=0$ in Definition \ref{def:equivariant-pair}, we find that all that is required is to show that 
  \begin{equation}
    \label{eq:problem-terms}
    \mu^2(\Phi^0_{2,L(n)}(a,b),x) - \mu^2(x,\Phi^0_{2,L}(a,b))
  \end{equation}
  is a coboundary; since the relevant complex has only a single degree this means to show that \eqref{eq:problem-terms} is zero.
  To interpret this expression, note that $CW^0(L,L(n))$ is a $CW^0(L(n),L(n))$--$CW^0(L,L)$ bimodule, and we are comparing the right action of $\Phi^0_{2,L}$ to the left action of $\Phi^0_{2,L(n)}$. Now we use the observation that $L$ and $L(n)$ are isomorphic in the wrapped Fukaya category, and in fact any pure generator (intersection point or chord) $x_0 \in CW^0(L,L(n))$ furnishes an isomorphism. Pick such an $x_0$. We obtain isomorphisms $\mu^2(-,x_0) : CW^0(L(n),L(n)) \to CW^0(L,L(n))$ and $\mu^2(x_0,-) : CW^0(L,L) \to CW^0(L,L(n))$, and by composition of the latter with the inverse of the former, an isomorphism $CW^0(L,L) \to CW^0(L(n),L(n))$. We claim this isomorphism coincides with $(\tau^n)_*$. Both maps are ring isomorphisms that preserves the relative $H_1(\C^*; \Z)$-grading, so one must only check that both isomorphisms map the generator $z_L$ to the generator $z_{L(n)}$ (each representing a chord that winds once around the cylinder). This is obvious for $(\tau^n)_*$, and for the other map it follows from a direct computation of the triangle products $\mu^2(x_0,z_L)$ and $\mu^2(z_{L(n)},x_0)$, which are equal. Thus, for $y \in CW^0(L,L)$, we have
  \begin{equation}
    \mu^2((\tau^n)_*y, x) = \mu^2(x,y)
  \end{equation}
  Applying this with $y = \Phi^0_{2,L}(a,b)$, and applying Lemma \ref{lem:same-phi}, we see that \eqref{eq:problem-terms} is zero.
\end{proof}

The preceding proposition justifies our choices of $c_L = 0$ and $c_{L(n)} = 0$, as with these choices, we do indeed obtain a representation $\rho = (\Phi^1_1)_{L,L(n)}$ of $SC^1(\C^*)$ on $CW^0(L,L(n))$. It remains to determine what representation of $SC^1(\C^*)$ or of $\fraksl_2$ is obtained this way. It is challenging to compute all of the moduli spaces involved in this action, but we can determine some of the structure geometrically. 

Before continuing, it will be useful to consider a more general situation into
which our pair $(L,L(n))$ falls. Suppose that $L_0$ and $L_1$ are exact
Lagrangian submanifolds in a Liouville domain $U$, possibly non-compact, which
are Lagrangian isotopic, although we do not require that the isotopy be
compactly supported. The pair $L$ have this property, since $L(n)$ is obtained
by wrapping $L$ at both ends. Note that this construction is not related to the
map $\tau^n$ considered above. The following proposition describes the
continuation element associated to an isotopy connecting $L_0$ to $L_1$, and
also a ``higher-order'' continuation element describing how the $SH^1(U)$
action changes.

\begin{proposition}
\label{prop:continuation}
  The isotopic Lagrangians $L_0$ and $L_1$ are isomorphic in the (non-equivariant) wrapped Fukaya category of $U$. Associated to an isotopy $\calL = \{L_t\}_{t\in [0,1]}$ taking $L_0$ to $L_1$, there is an element 
\begin{equation}
\kappa_{0,\calL} \in CW^0(L_0,L_1)
\end{equation} such that $\mu^2(\kappa_{0,\calL},-)$ induces an isomorphism 
\begin{equation}
  K_{\calL} : HW^*(L_0,L_0) \to HW^*(L_0,L_1)
\end{equation}
 There is a map 
\begin{equation}
\kappa_{1,\calL} : SC^1(U) \to CW^0(L_0,L_1)
\end{equation}
such that if we further assume that the Lagrangian $L_0$ satisfies $(\Phi^0_1)_{L_0} = 0$,
 \begin{equation}
   \label{eq:kappa-1}
  K_{\calL}\circ (\Phi^1_1)_{L_0,L_0}(\xi) - (\Phi^1_1)_{L_0,L_1}(\xi)\circ K_{\calL} + \mu^2(\kappa_{1,\calL}(\xi),-) = \text{terms involving $\mu^1$ or $d$}
\end{equation}
If we further assume that $L_0$ is connected and simply connected (implying $L_1$ is as well), so that the wrapped Floer groups carry relative $H_1(U;\Z)$ gradings, then the element $\kappa_{0,\calL}$ is homogeneous with respect to the relative $H_1(U;\Z)$ grading, and the maps $K_{\calL}$ and $\kappa_{1,\calL}$ are homogeneous with respect to the relative $H_1(U;\Z)$ gradings.
\end{proposition}
\begin{proof}
  The element $\kappa_{0,\calL}$ and the map $K_{\calL}$ are the standard continuation element and map respectively. The element $\kappa_{0,\calL}$ is defined by counting disks with a moving boundary condition determined by the isotopy $\calL$. It in fact determines isomorphisms
  \begin{equation}
    \mu^2(\kappa_{0,\calL},-) : HW^*(L,L_0) \to HW^*(L,L_1)
  \end{equation}
  for any $L$, and $K_{\calL}$ is the case $L = L_0$. The map $K_{\calL}$ is also determined by counting strips with $L_0$ on one side, and a moving boundary condition determined by $\calL$ on the other. 

Because it is defined by counting maps of a disk to $U$, the element
$\kappa_{0,\calL}$ is homogeneous with respect to the relative grading. Indeed,
any chord from $L_0$ to $L_1$ that appears in $\kappa_{0,\calL}$ is homotopic
via the holomorphic disk to a path $\gamma : [0,1]\to U$ such that $\gamma(t)
\in L_t$. Post-composing these paths with the reverse isotopy yields a path
$\bar{\gamma} : [0,1]\to U$ such that $\bar{\gamma}(t) \in L_0$ for all $t$.
Via this process, the grading difference between two chords contributing to
$\kappa_{0,\calL}$ is measured by a loop in $L_0$. Since $L_0$ is assumed
simply connected the difference must be zero.

The proof of homogeneity the map $K_{\calL}$ is similar. We use the fact that the isotopy $\calL$ determines a bijection between homotopy classes of paths from $L_0$ to $L_0$ and homotopy classes of paths from $L_0$ to $L_1$. Each strip contributing to $K_{\calL}$ then witnesses that the input and output have gradings that are related by that bijection.

The element $\kappa_{1,\calL} : SC^1(U) \to CW^0(L_0,L_1)$ is defined by counting disks with one input and one boundary puncture, and with a moving boundary condition determined by $\calL$. The presence of the moving boundary condition changes the number of degrees of freedom in the domain, and we consider a one-dimensional parameter space where the interior puncture is allowed to move along a horizontal line in the domain. This gives $\kappa_{1,\calL}$ degree $-1$ as desired. The interior puncture is asymptotic to an orbit $\xi \in SC^1(U)$. To prove the relation \eqref{eq:kappa-1}, we consider a one-dimensional moduli space of strips with an interior puncture, one boundary condition constant on $L_0$, and the other determined by $\calL$. The space of domains has dimension two in this case, parametrized by the position of the interior puncture. The degenerations where the interior puncture collides with one of the boundary punctures gives the terms involving $K_{\calL}$. The degeneration where the interior puncture approaches the boundary with the moving boundary condition yields the term $\mu^2(\kappa_{1,\calL}(\xi),-)$. When the interior puncture approaches the boundary with the constant $L_0$ condition, we obtain a degeneration combining $\Phi^0_1 \in CW^1(L_0,L_0)$ with a disk that resembles $\mu^2$ but has a moving boundary condition on one edge. By hypothesis, we can discard these terms, and all other boundary components either the differential $\mu^1$ on $CW^*(L_0,L_0)$ or $CW^*(L_0,L_1)$ or the differential $d$ on $SC^*(U)$.

The last assertion is that $\kappa_{1,\calL}$ is homogeneous with respect to the absolute grading on $SC^1(U)$ and the relative grading on $CW^0(L_0,L_1)$. This means we regard both gradings as relative, and claim that $\kappa_{1,\calL}$ preserves grading differences. Again, this follows from the topology of the surfaces used in the definition. 
\end{proof}

\begin{remark}
An important observation connecting the previous proposition to the general theory is that the element $\kappa_{1,\calL}$ can be identified with the cocycle measuring the difference between the action of $SC^1(U)$ on $CW^0(L_0,L_0)$ and the action on $CW^0(L_0,L_1)$.
\end{remark}

We now consider the action of the Cartan subalgebra $H^1(\C^*) \subset SC^1(\C^*)$, spanned by the generator  $\xi_0$.
 
In the case of $L$ and $L(n)$ in $\C^*$, each component of $CW^0(L,L(n))$ with respect to the relative $H_1(\C^*;\Z)$-grading has rank one. As $\xi_0$ is represented by a contractible loop, its action preserves the $H_1(\C^*;\Z)$-grading, and so acts diagonally on the basis of chords. The following proposition gives information about the eigenvalues of $\Phi^1_1(\xi_0)$. This is a special case of a result due to Nick Sheridan, although our proof is particular to our special case.
\begin{proposition}
  Let $\lambda(x)$ denote the eigenvalue of $\Phi^1_1(\xi_0)$ on the generator $x \in CW^0(L,L(n))$. Consider two generators $x_1,x_2$, and denote by $[x_1]-[x_2] \in H_1(\C^*;\Z)$ the relative grading difference between them. We have
  \begin{equation}
    \lambda(x_1) - \lambda(x_2) = \langle \xi_0,[x_1]-[x_2]\rangle
  \end{equation}
  where the right hand side denotes the pairing between $\xi_0$, thought of as an element of $H^1(\C^*) \subset SH^1(\C^*)$, with a class in $H_1(\C^*;\Z)$.\footnote{Dangerous bend: the class of $\xi_0$ in $H^1(\C^*)$ is not zero even though $\xi_0$ is represented in symplectic cohomology by a contractible loop.}
\end{proposition}
\begin{proof}
  First we observe that the statement is true for the representation $CW^0(L,L)$ considered previously. The generator $x = z^n$ satisfies $\lambda(z^n) = n$, and also the relative grading of $z^{n_1}$ and $z^{n_2}$ is $(n_1 - n_2)[S^1] \in H_1(\C^*;\Z)$. Since $\langle[\xi_0],[S^1]\rangle = 1$, both sides reduce to $n_1 - n_2$. Also observe that, for $CW^0(L,L)$, the relative grading by $H_1(\C^*;\Z)$ may be enhanced to an absolute grading, by declaring the identity element $1$ to have grading zero. This absolute grading corresponds to the homology classes of Reeb chords starting and ending on $L$.

Now we use the continuation elements defined in Proposition \ref{prop:continuation}. Let $\calL$ be an isotopy between $L_0 = L$ and $L_1 = L(n)$, and $\kappa_{0,\calL} \in CW^0(L,L(n))$ and $\kappa_{1,\calL} : SC^1(\C^*) \to CW^0(L,L(n))$ be the continuation elements. Since $\xi_0$ is a contractible loop, we find that $\kappa_{1,\calL}(\xi_0)$ and $\kappa_{0,\calL}$ lie in the same graded component of $CW^0(L,L(n))$. Since the graded components are one-dimensional, we have a proportionality
\begin{equation}
  \kappa_{1,\calL}(\xi_0) = \epsilon \kappa_{0,\calL}
\end{equation}
for some $\epsilon \in \K$.

Now equation \eqref{eq:kappa-1} reads 
\begin{equation}
  K_{\calL}\circ (\Phi^1_1)_{L_0,L_0}(\xi) - (\Phi^1_1)_{L_0,L_1}(\xi)\circ K_{\calL} + \mu^2(\epsilon \kappa_{0,\calL},-) = 0
\end{equation}
Plug in the element $z^n \in CW^0(L,L)$ into the equation to obtain
\begin{equation}
  nK(z^n) - (\Phi^1_1)_{L,L(n)}(\xi_0)(K(z^n)) + \epsilon K(z^n) = 0
\end{equation}
This shows that $(\Phi^1_1)_{L,L(n)}(\xi_0)$ has eigenvalue $n + \epsilon$ on $K(z^n)$. 

To conclude, the difference between the eigenvalues of $(\Phi^1_1)_{L,L(n)}(\xi_0)$ on $K(z^{n_1})$ and $K(z^{n_2})$ is once again $n_1 - n_2$. Because the map $K$ preserves the relative $H_1(\C^*;\Z)$ grading, the expression $\langle \xi_0, [K(z^{n_1})] - [K(z^{n_2})] \rangle$ also equals $n_1 - n_2$.
\end{proof}
\begin{corollary}
  The weight spaces of $CW^0(L,L(n))$, under the action of $\Phi^1_1(\xi_0)$, are one-dimensional.
\end{corollary}

Next, we consider the generator $\xi_1 \in SC^1(\C^*)$. Geometrically, this generator lies in one end of $\C^*$. This end also contains some of the chords contributing to $CW^0(L,L(n))$. Let $v_+ \in CW^*(L,L(n))$ denote the intersection generator closest to this end, so that all other generators further into the end are proper chord generators. Symmetrically, considering the action $\xi_{-1}$ on the opposite end of $\C^*$, we find an intersection generator $v_-$ that is closest to that end.

\begin{proposition}
  \label{prop:highest-weight}
  We have $\Phi^1_1(\xi_1)(v_+) = 0$, and $\Phi^1_1(\xi_{-1})(v_-) = 0$.
\end{proposition}

These statements are consequences of the following more general proposition. To set this up, we note that on the ends, $L$ and $L(n)$ coincide, and so if we restrict our attention to generators on a single end, there is a bijection between chords from $L$ to $L(n)$ and chords from $L$ to $L$ at that end. 

In what follows, we will actually use a perturbed copy of $L$. This $\tilde{L}$
is obtained by pushing $L$ off of itself by a small amount in the direction of
the Hamiltonian flow, as shown at the top of figure \ref{fig:neck}. Thus $\tilde{L}$ and
$L$ intersect in one point, which has degree zero as a morphism from
$\tilde{L}$ to $L$. Because $\tilde{L}$ is a small push off of $L$ by the
wrapping Hamiltonian itself, the Floer complex $CW^*(\tilde{L},L)$ is naturally
identified with $CW^*(L,L)$. The chords are essentially the same, they are just
slightly shorter in $CW^*(\tilde{L},L)$. The continuation element $\kappa_0 \in
CW^0(\tilde{L},L)$ from Proposition \ref{prop:continuation}  is given by the unique intersection
point. This identification is also compatible with the closed-open string maps
such as $\Phi^1_1$, so the $SC^1(\C^*)$ action on $CW^0(\tilde{L},L)$
corresponds to the action $CW^*(L,L)$, and both are the action of vector fields
on functions in the mirror interpretation.

The continuation element $\kappa_0$ also relates $CW^*(L,L(n))$ to
$CW^*(\tilde{L},L(n))$, and the action of $SC^1(\C^*)$ on $CW^*(L,L(n))$
corresponds to the action on $CW^*(\tilde{L},L(n))$.

Having done the perturbation, we will also modify the Hamiltonian slightly, so
that there is a region in the middle of the $\C^*$ where no wrapping occurs.
This region should contain all of the intersection points between $\tilde{L}$
and $L(n)$.

Comparing the two pairs $(\tilde{L},L)$ and $(\tilde{L},L(n))$, we see that the
ends of these pictures resemble each other. Thus there are some correspondences
for certain chords for the pair $(\tilde{L},L)$ with certain chords for the
pair $(\tilde{L},L(n))$. To spell this out, we denote by $\zeta_k$ the chord
that winds $k$ times around the cylinder, where $\zeta_0$ is the intersection
point, and $\zeta_k$ lies in the right-hand end of the figure. For the pair
$(\tilde{L},L(n))$, we have the intersection point $v_+$ that is right-most in
the figure, and further to the right of it we find chords that we denote by
$v_{+,k}$, for $k > 0$, which wind $k$ times around the cylinder, and write
$v_{+,0} = v_+$ We can set up a bijection between these two sets of generators
by mapping $\zeta_k$ to $v_{+,k}$ for $k \geq 0$. We call this bijection $\flat_+$.
In symmetrical fashion, we have generators $v_{-} = v_{-,0}$ and $v_{-,k}$ for
$k \leq 0$ of $CW^0(\tilde{L},L(n))$ in the other end of $\C^*$, and we can set
up a bijection between these and the chords $\zeta_k$ for $k \leq 0$. We call
this bijection $\flat_-$.\footnote{These bijections should not be regarded as parts
of a single correspondence between generators of $CW^0(\tilde{L},L)$ and
$CW^0(\tilde{L},L(n))$.}

\begin{figure}[!h]
\label{fig:neck} \centering
\includegraphics[scale=2.2]{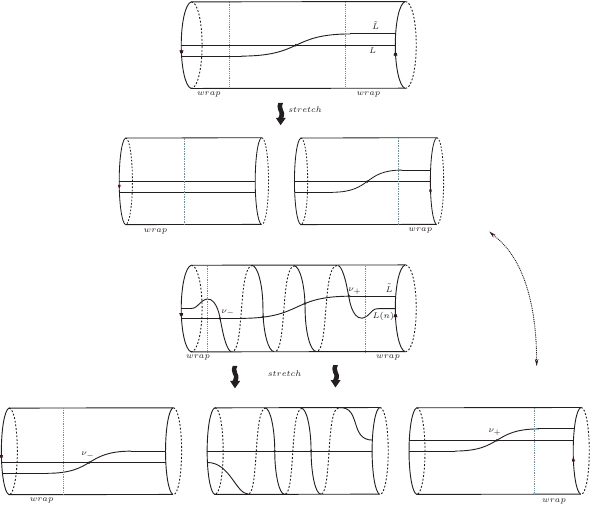}  \caption{Stretching the neck to compute the action.} 
\label{fig:} \end{figure}

\begin{proposition}
  Suppose $n > 0$.
  For $j > 0, k\geq 0$, the action of $(\Phi^1_1)_{\tilde{L},L(n)}(\xi_j)$ on the generators $v_{+,k}$ corresponds to the action of $(\Phi^1_1)_{\tilde{L},L}(\xi_j)$ on the generators $\zeta_k$. That is
  \begin{equation}
    (\Phi^1_1)_{\tilde{L},L(n)}(\xi_j)(v_{+,k}) = \flat_+((\Phi^1_1)_{\tilde{L},L}(\xi_j)(\zeta_k)) = k\cdot v_{+,k+j} 
  \end{equation}
  Symmetrically, for $j < 0, k \leq 0$, the action of $(\Phi^1_1)_{\tilde{L},L(n)}(\xi_j)$ on $v_{-,k}$ corresponds to the action of $(\Phi^1_1)_{L,L}(\xi_j)$ on $v_{-,k}$:
  \begin{equation}
    (\Phi^1_1)_{\tilde{L},L(n)}(\xi_j)(v_{-,k}) = \flat_-((\Phi^1_1)_{\tilde{L},L}(\xi_j)(\zeta_k)) = k\cdot v_{-,k+j} 
  \end{equation}
\end{proposition}
\begin{proof}
We first consider the statement for $j > 0$, $k \geq 0$.  The key point to be justified is that the pseudo-holomorphic curves contributing to $(\Phi^1_1)_{\tilde{L},L(n)}(\xi_j)$ are contained in the end where $\xi_j$, $v_{+,k}$ and $v_{+,k+j}$ lie, while the curves contributing to $(\Phi^1_1)_{\tilde{L},L}(\xi_j)$ are entirely contained in the end where $\xi_j$, $\zeta_k$, $\zeta_{k+j}$ lie. Once this is shown, we know that the counts of maps must match because the configuration of Lagrangians in these regions is geometrically the same. 

To establish this, we will apply a neck-stretching deformation of the problem, and show that in the limit, the curves are entirely contained in the end. This implies that there is some finite deformation having this property.

The neck-stretching deformation we use for the pair $(\tilde{L},L(n))$ involves stretching off the ends of the manifold, including the generators $v_+$ and $v_-$ in their respective ends. This splits the target $\C^*$ into three copies of $\C^*$, an end $U_+$ containing $v_+$, an end $U_-$ containing $v_-$ and the interior $U_0$ containing all the other intersection points, as shown in the bottom half of figure \ref{fig:neck}. (This is only possible if $n > 0$.) Of course, the neck-stretching process may cause parts of the pseudo-holomorphic curves contributing to $(\Phi^1_1)_{\tilde{L},L}(\xi_j)(v_{+,k})$ to break off and remain in the interior. The sort of components that may appear in the interior are either holomorphic planes asymptotic to Reeb orbits, or holomorphic half-planes with boundary on one of the Lagrangians $L$ or $L(n)$ that are asymptotic to Reeb chords. Both types of curves do not exist in $U_0$ for topological reasons: the former because the Reeb orbits are not contractible, the latter because the Lagrangians are simply connected and the Reeb chords are not contractible. Thus there is a finite point in the neck-stretching process where all holomorphic curves are contained in the relevant end.

The deformation for the pair $(\tilde{L},L)$ is asymmetrical. We stretch along one circle splitting the cylinder into two parts $W_0$ and $W_+$, so that the generators $\zeta_k$ for $k \geq 0$ all end up in $W_+$ in the limit, while the chords $\zeta_k$ for $k < 0$ end up in $W_0$, as shown in the top half of figure \ref{fig:neck}. Once again, the limit configurations of curves contributing to $(\Phi^1_1)_{\tilde{L},L}(\xi_j)(\zeta_k)$ for $k \geq 0$. can have no components in $W_0$ for topological reasons. And thus there is a point in the neck-stretching process where all curves are contained in the relevant end.

Finally, comparing the pictures in $U_+$ and $W_+$, we find that they are the same, showing that the action of $\xi_j$ on the $v_{+,k}$ generators is identified with the action on the $\zeta_k$ generators, for $k \geq 0$.

In order to analyze the case where $j < 0, k \leq 0$, we could apply a different deformation to the pair $(\tilde{L},L)$, splitting the cylinder into $W'_-$ and $W'_0$, so that the generators $\zeta_k$ for $k \leq 0$ end up in $W'_-$. The rest of the argument follows \emph{mutatis mutandis}. Alternatively, we can use the automorphism of the cylinder that switches the two ends and maps $\tilde{L}$, $L$, and $L(n)$ to themselves. This brings us back into the previous case. This gives the desired result, once we realize that this automorphism actually maps $\xi_j$ to $-\xi_{-j}$, due to the fact that the definition of $\xi_j$ requires an orientation on the core $S^1 \subset \C^*$, which is reversed by the automorphism. 

\end{proof}

It remains to tie these propositions together. Let $V(n) \subset CW^0(L,L(n))$
denote the space spanned by the intersection generators, and consider the
action of $\fraksl_2 = \langle \xi_{-1},\xi_0,\xi_1\rangle $ on this space. We
regard $\xi_0$ as the grading operator, $\xi_1$ as the raising operator, and
$\xi_{-1}$ as the lowering operator. Proposition \ref{prop:highest-weight} says
that $v_+$ is a highest-weight vector for $\fraksl_2$, implying that the span
of all the weight spaces below this one is stable under $\fraksl_2$.
Symmetrically, Proposition \ref{prop:highest-weight} says that $v_-$ is a
lowest-weight vector for $\fraksl_2$, and we conclude that the span of the
weight spaces between $v_+$ and $v_-$ is stable under $\fraksl_2$. This span
is precisely $V(n)$. Furthermore, because all the weight spaces are one
dimensional, and the total dimension is $n+1$, we conclude that the
representation is irreducible. Thus we have proved the following.

\begin{theorem}
  The subspace $V(n) \subset CW^0(L,L(n))$ spanned by the intersection generators is stable under the action of $\fraksl_2$ via $(\Phi^1_1)_{L,L(n)}$. It is an irreducible representation of $\fraksl_2$ isomorphic to the representation in homogeneous polynomials of degree $n$ in two variables.
\end{theorem}

\bibliographystyle{amsplain}
\bibliography{equiv}

\end{document}